\documentclass[reqno,a4paper]{amsart}
\usepackage[english]{babel}

\usepackage{graphicx,mathrsfs,tikz,latexsym,ifthen,amsmath,amsfonts,amssymb,amsthm,stmaryrd,fancyhdr,amscd,amsbsy,amstext,tensor, tikz-cd}
\usepackage{enumerate,enumitem,empheq,mathtools,mathpazo,verbatim,a4wide,nomencl} 
\mathtoolsset{showonlyrefs,showmanualtags}

\usepackage{epigraph}

\usepackage[utf8x]{inputenc}

\usepackage[toc,page]{appendix}
\usepackage[mathscr]{euscript} 
\allowdisplaybreaks

\usetikzlibrary{positioning,shapes,shadows,arrows}

\usepackage{color}
\definecolor{MyDarkBlue}{rgb}{0.15,0.25,0.45}

\newcommand{\triend}{\parbox{2mm}{\hfill} \hfill\text{\hspace{0.2mm}}\hfill$\triangle$}
\newcommand{\ocend}{\parbox{2mm}{\hfill} \hfill\text{\hspace{0.2mm}}\hfill$\oslash$}

\usepackage{hyperref}

\usepackage[hyperpageref]{backref} 

\newtheorem{theorem}{Theorem}
\newtheorem{proposition}[theorem]{Proposition}
\newtheorem{lemma}[theorem]{Lemma}
\newtheorem{corollary}[theorem]{Corollary}

\newtheorem{corollary*}{Corollary}
\newtheorem*{theorem*}{Theorem}
\newtheorem*{proposition*}{Proposition}
\newtheorem*{conjecture*}{Conjecture}

\numberwithin{equation}{section}
\numberwithin{theorem}{section}

\theoremstyle{remark}
\newtheorem{ex}[theorem]{Example}
\newenvironment{example}{\begin{ex}}{\triend\end{ex}}

\theoremstyle{remark}
\newtheorem{rem}[theorem]{Remark}
\newenvironment{remark}{\begin{rem}}{\triend\end{rem}}

\theoremstyle{definition}
\newtheorem{defin}[theorem]{Definition}
\newenvironment{definition}{\begin{defin}}{\ocend\end{defin}}

\newcommand{\fun}[1]{\mathbb{F}(#1)} 

\newcommand{\cf}{\mathbb{1}}

\newcommand{\Int}{\mathsf{Int}}

\newcommand{\calE}{\mathcal{E}}
\newcommand{\calF}{{\mathcal F}}

\newcommand{\calQ}{\mathcal{Q}}
\newcommand{\calM}{\mathcal{M}}
\newcommand{\calO}{\mathcal{O}}

\newcommand{\fraksl}{\mathfrak{sl}}
\newcommand{\hatfraksl}{\widehat{\mathfrak{sl}}}
\newcommand{\frakg}{\mathfrak{g}}

\newcommand{\F}{\mathbb{F}}
\newcommand{\K}{{\mathbb{K}}}
\newcommand{\Q}{{\mathbb{Q}}}
\newcommand{\R}{{\mathbb{R}}}
\newcommand{\Z}{{\mathbb{Z}}}
\newcommand{\C}{{\mathbb{C}}}
\newcommand{\N}{\mathbb{N}}
\newcommand\bfH{\mathbf{H}}

\newcommand{\bfU}{\mathbf{U}}

\newcommand{\sfK}{\mathsf{K}}
\newcommand{\Rep}{\mathsf{Rep}}

\hypersetup{anchorcolor=black, linktocpage=true,
colorlinks=true,
citecolor=MyDarkBlue,
linkcolor=MyDarkBlue,
urlcolor=black,
pdfauthor={Francesco Sala and Olivier Schiffmann},
pdftitle={Fock space representation of the circle quantum group}, 
breaklinks=true,
plainpages=true
}

\title[Fock space representation of the circle quantum group]{Fock space representation of the circle quantum group}

\author[F.~Sala]{Francesco Sala}
\address[Francesco Sala]{Università di Pisa, Dipartimento di Matematica, Largo Bruno Pontecorvo 5, 56127 Pisa (PI), Italy}
\address{Kavli IPMU (WPI), UTIAS, The University of Tokyo, Kashiwa, Chiba 277-8583, Japan}
\curraddr{}
\email{\href{mailto:francesco.sala@unipi.it}{francesco.sala@unipi.it}}

\author[O.~Schiffmann]{Olivier Schiffmann}
\address[Olivier Schiffmann]{Laboratoire de Math\'ematiques, Faculté des Sciences d'Orsay, Université Paris-Sud, B\^at. 307, 91405 Orsay Cedex, France}
\curraddr{}
\email{\href{mailto:olivier.schiffmann@math.u-psud.fr}{olivier.schiffmann@math.u-psud.fr}}

\thanks{The work of the first-named author is partially supported by World Premier International Research Center Initiative (WPI), MEXT, Japan, by JSPS KAKENHI Grant number JP17H06598 and by JSPS KAKENHI Grant number JP18K13402.}

\begin{document}

\begin{flushright}
IPMU--19--0027
\end{flushright}

\vskip 1cm

\begin{abstract}
	In \cite{art:salaschiffmann2017} we have defined quantum groups $\bfU_\upsilon(\fraksl(\R))$  and $\bfU_\upsilon(\fraksl(S^1))$, which can be interpreted as continuum generalizations of the quantum groups of the Kac-Moody Lie algebras of finite, respectively affine type $A$. In the present paper, we define the Fock space representation $\calF_\R$ of the quantum group $\bfU_\upsilon(\fraksl(\R))$ as the vector space generated by real pyramids (a continuum generalization of the notion of partition). In addition, by using a variant version of the ``folding procedure" of Hayashi-Misra-Miwa, we define an action of $\bfU_\upsilon(\fraksl(S^1))$ on $\calF_\R$. 
\end{abstract}

\maketitle\thispagestyle{empty}

\tableofcontents

\section{Introduction}

The present article continues the study of the \textit{circle quantum group} $\bfU_\upsilon(\fraksl(S^1))$ introduced in \cite{art:salaschiffmann2017}, where $S^1\coloneqq \R/\Z$. In that paper we constructed geometrically a family of representations $V_{g,\,r}$ indexed by a pair of positive integers $(g,r)$\footnote{In the setting of \cite{art:salaschiffmann2017}, $\bfU_\upsilon(\fraksl(S^1))$ is realized as (reduced) Drinfeld double of the spherical Hall algebra of parabolic torsion sheaves on a genus $g$ curve. The representation $V_{g,\, r}$ arises from the Hall algebra of rank $r$ parabolic vector bundles on the curve.}, with $V_{0,1}$ being the natural ``vector'' representation and $V_{0,\, r}$ an analog of the $r$-fold symmetric power of $V_{0,\, 1}$. Our goal here is to define combinatorially a Fock space representation $\calF_\R$ of $\bfU_\upsilon(\fraksl(S^1))$, which may be thought of as a limit of $V_{0,\, r}$ as $r$ tends to infinity.  

Before stating the main results of the paper, let us briefly recall the classical construction of the Fock space representation $\calF$ of $\bfU_\upsilon(\fraksl(\infty))$ and $\bfU_\upsilon(\hatfraksl(n))$. Set $\widetilde{\Q} \coloneqq\Q[\upsilon, \upsilon^{-1}]$. Then, $\calF$ is the $\widetilde{\Q}$-vector space with a basis formed by vectors $\vert \lambda\rangle$ labelled by \emph{partitions} $\lambda$, i.e., 
\begin{align}
\calF\coloneqq  \bigoplus_{\lambda}\,  \widetilde \Q \vert \lambda\rangle\ .
\end{align}
Roughly speaking, the action of the positive (resp.\ negative) part of $\bfU_\upsilon(\fraksl(\infty))$ is given in terms of the combinatorial procedure of removing (resp.\ adding) a box to the Young diagram $Y_\lambda$ associated to the partition $\lambda$, while the Cartan subalgebra acts diagonally with a $\upsilon$-factor depending on the number of addable and removable boxes of a Young diagram. 

The action of $\bfU_\upsilon(\hatfraksl(n))$ on $\calF$ was originally defined by Hayashi \cite{art:hayashi1990} by using $\upsilon$-analogs of Clifford and Weyl algebras and their actions on the exterior and polynomial algebras, respectively. Miwa-Misra \cite{art:misramiwa1990} gave another interpretation of the action in terms of operations on Young diagrams. Finally, Varagnolo-Vasserot \cite{art:varagnolovasserot1999} reinterpreted the ``folding procedure" in terms of Hall algebras associated to nilpotent representations of the infinite quiver $A_\infty$ and the affine quiver $A_{n-1}^{(1)}$ respectively. Note that $\calF$ is the level one Fock space with fundamental weight $\Lambda_0$ of $\bfU_\upsilon(\hatfraksl(n))$.

\subsection{Main results }

Just as the quantum group $\bfU_\upsilon(\fraksl(S^1))$ is an uncountable colimit of $\bfU_\upsilon(\hatfraksl(n))$ as $n$ tends to infinity, the Fock space $\calF_{\R}$ is an uncountable colimit of the standard Fock space representation of $\bfU_\upsilon(\hatfraksl(n))$. One main novelty of the limit which we consider is that instead of partitions, the Fock space $\calF_{\R}$ has a basis indexed by what we call \emph{(real) pyramids}\footnote{Although we are confident that we have chosen the right terminology, one of the anonymous referees pointed out that \emph{ziggurat} could have been a good choice as well.} (see Section~\ref{s:partitionpyramids}). Integral pyramids are close cousins of Maya diagrams and are in bijection with partitions. For instance, on the left-hand-side there is the Young diagram of the partition $(5,4, 4,3,1, 1)$ with its standard contents written inside the corresponding boxes, and on the right-hand-side there is the corresponding integral pyramid:
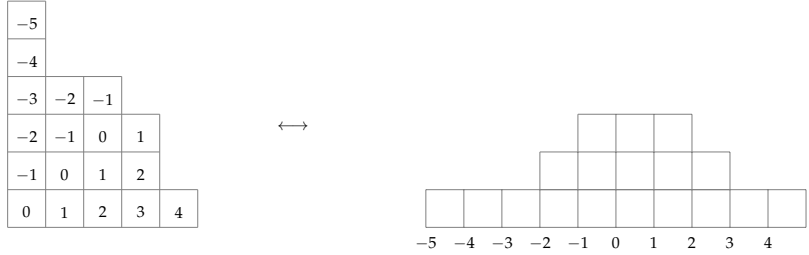
\begin{align}
	\begin{tikzpicture}
	\draw[gray,very thin](-6,0) -- (-3.5,0);
	\draw[gray,very thin](-6,0.5) -- (-3.5,0.5);
	\draw[gray,very thin](-6,1) -- (-4,1);
	\draw[gray,very thin](-6,1.5) -- (-4,1.5);
	\draw[gray,very thin](-6,2) -- (-4.5,2);
	\draw[gray,very thin](-6,2.5) -- (-5.5,2.5);
	\draw[gray,very thin](-6,3) -- (-5.5,3);
	\draw[gray,very thin](-6,0) -- (-6,3);
	\draw[gray,very thin](-5.5,0) -- (-5.5,3);
	\draw[gray,very thin](-5,0) -- (-5,2);
	\draw[gray,very thin](-4.5,0) -- (-4.5,2);
	\draw[gray,very thin](-4,0) -- (-4,1.5);
	\draw[gray,very thin](-3.5,0) -- (-3.5,0.5);
	\draw[step=0.5cm,gray,very thin] (-0.5,0) grid (4.5,0.5);
	\draw[step=0.5cm,gray,very thin] (1,0.5) grid (3.5,1);
	\draw[step=0.5cm,gray,very thin] (1.5,1) grid (3,1.5);
	\draw[gray,very thin] (1,0.5) -- (1,1);
	\draw[gray,very thin] (1.5,1) -- (1.5,1.5);
	\draw (2,0) circle (0) node[anchor=north] {\tiny{$0$}};
	\draw (2.5,0) circle (0) node[anchor=north] {\tiny{$1$}};
	\draw (1.5,0) circle (0) node[anchor=north] {\tiny{$-1$}};
	\draw (3,0) circle (0) node[anchor=north] {\tiny{$2$}};
	\draw (1,0) circle (0) node[anchor=north] {\tiny{$-2$}};
	\draw (3.5,0) circle (0) node[anchor=north] {\tiny{$3$}};
	\draw (0.5,0) circle (0) node[anchor=north] {\tiny{$-3$}};
	\draw (4,0) circle (0) node[anchor=north] {\tiny{$4$}};
	\draw (0,0) circle (0) node[anchor=north] {\tiny{$-4$}};
	\draw (-0.5,0) circle (0) node[anchor=north] {\tiny{$-5$}};
	\draw (-2.25,1.5) circle (0) node[anchor=north] {\tiny{$\longleftrightarrow$}};
	\draw (-5.75,0.4) circle (0) node[anchor=north] {\tiny{$0$}};
	\draw (-5.25,0.9) circle (0) node[anchor=north] {\tiny{$0$}};
	\draw (-4.75,1.4) circle (0) node[anchor=north] {\tiny{$0$}};
	\draw (-5.25,0.4) circle (0) node[anchor=north] {\tiny{$1$}};
	\draw (-4.75,0.9) circle (0) node[anchor=north] {\tiny{$1$}};
	\draw (-4.25,1.4) circle (0) node[anchor=north] {\tiny{$1$}};
	\draw (-5.75,0.9) circle (0) node[anchor=north] {\tiny{$-1$}};
	\draw (-5.25,1.4) circle (0) node[anchor=north] {\tiny{$-1$}};
	\draw (-4.75,1.9) circle (0) node[anchor=north] {\tiny{$-1$}};
	\draw (-5.75,1.4) circle (0) node[anchor=north] {\tiny{$-2$}};
	\draw (-5.75,1.9) circle (0) node[anchor=north] {\tiny{$-3$}};
	\draw (-5.75,2.4) circle (0) node[anchor=north] {\tiny{$-4$}};
	\draw (-5.75,2.9) circle (0) node[anchor=north] {\tiny{$-5$}};
	\draw (-5.25,1.9) circle (0) node[anchor=north] {\tiny{$-2$}};
	\draw (-4.75,0.4) circle (0) node[anchor=north] {\tiny{$2$}};
	\draw (-4.25,0.4) circle (0) node[anchor=north] {\tiny{$3$}};
	\draw (-3.75,0.4) circle (0) node[anchor=north] {\tiny{$4$}};
	\draw (-4.25,0.9) circle (0) node[anchor=north] {\tiny{$2$}};
	\end{tikzpicture}
\end{align}
Pictorially, we pass from Young diagrams to $\Z$-pyramids by tilting $45$ degrees to the left (or writing partitions the Russian way) and letting gravity act. Unlike partitions, pyramids admit a natural extensions to $\R$. For example, the following is a real pyramid:
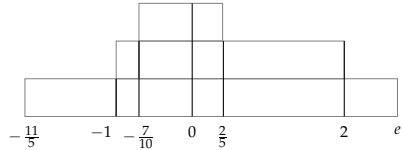
\begin{align}
	\begin{tikzpicture}
	\draw [gray,very thin] (-0.5,0) rectangle (0.7,0.5);
	\draw [gray,very thin] (0.7,0.5) rectangle (1,1);
	\draw [gray,very thin] (0.7,0) rectangle (4.4,0.5);
	\draw [gray,very thin] (1,1) rectangle (2.1,1.5);
	\draw [gray,very thin] (1,0.5) rectangle (3.7,1);
	\draw[] (1.7,0) -- (1.7,1.5);
	\draw[] (0.7,0) -- (0.7,0.5);
	\draw[] (1,0) -- (1,1);
	\draw[] (2.11,0) -- (2.11,1);
	\draw[] (3.7,0) -- (3.7,1);
	\draw (1.7,0) circle (0) node[anchor=north] {\tiny{$0$}};
	\draw (-0.5,0) circle (0) node[anchor=north] {\tiny{$-\frac{11}{5}$}};
	\draw (0.5,0) circle (0) node[anchor=north] {\tiny{$-1$}};
	\draw (1,0) circle (0) node[anchor=north] {\tiny{$-\frac{7}{10}$}};
	\draw (2.1,0) circle (0) node[anchor=north] {\tiny{$\frac{2}{5}$}};
	\draw (3.7,0) circle (0) node[anchor=north] {\tiny{$2$}};
	\draw (4.4,0) circle (0) node[anchor=north] {\tiny{$e$}};
	\end{tikzpicture}
\end{align}
Let $\mathsf{Pyr}(\R)$ be the set of all $\R$-pyramids. As a vector space, our Fock space is defined as
\begin{align}
	\calF_\R\coloneqq  \bigoplus_{p \in \mathsf{Pyr}(\R)}\,  \widetilde \Q \vert p\rangle \ .
\end{align}
In \cite{art:salaschiffmann2017} we also defined a Lie algebra $\fraksl(\R)$ --- now as a uncountable colimit of the Lie algebras $\fraksl(n)$ as $n$ tends to $\infty$ --- and a corresponding quantum group $\bfU_\upsilon(\fraksl(\R))$, the \emph{line quantum group}. It is generated by elements $E_J,F_J,K_J$ for $J$ an interval (cf.\ Definition~\ref{def:R}). Let $ 0 $ stand for the empty pyramid.
\begin{theorem*}[cf.\ Theorem~\ref{thm:Fock-R}]
		The following formulas define an action of the quantum group $\bfU_\upsilon(\fraksl(\R))$ on $\calF_\R$: for any $J=[a,b[$ and $p \in \mathsf{Pyr}(\R)$
	\begin{align}
	E_J \vert p\rangle&\coloneqq \begin{cases}
	-\upsilon^{1/2} (-\upsilon)^{p(b)-p(a)}\, \vert p-\cf_J\rangle &\text{if $J$ is a removable interval of $p$}\ ,\\[4pt]
	0 & \text{otherwise}\ ,
	\end{cases}\\[6pt]
	F_J \vert p\rangle&\coloneqq \begin{cases}
	\upsilon^{1/2} (-\upsilon)^{p(a)-p(b)} \, \vert p+\cf_J\rangle &\text{if $J$ is an addable interval of $p$}\ ,\\[4pt]
	0 & \text{otherwise}\ ,
	\end{cases}\\[6pt]
	K_J \vert p\rangle&\coloneqq \upsilon^{n_{J}(p)}\, \vert p\rangle\ ,
	\end{align}
	where
	\begin{align}
	n_{J}(p) =
	\begin{cases}
	0 & \text{if $J$ is neither addable or removable to $p$}\ , \\
	1 & \text{if $J$ is addable to $p$}\ ,\\
	-1 & \text{if $J$ is removable to $p$}\ .
	\end{cases} 
	\end{align}
	The representation $\calF_{\R}$ is an irreducible highest weight representation with highest weight vector $\vert 0 \rangle$. Here, $\cf_J$ is the characteristic function of the interval $J$.
\end{theorem*}
As one may see above, the continuum analog of removing or adding a box in a Young diagram is the operation of removing or adding intervals in a pyramid. 

The circle quantum group $\bfU_\upsilon(\fraksl(S^1))$ was defined in \cite{art:salaschiffmann2017}.  Due to the absence of simple roots for the circle Lie algebra $\fraksl(S^1)$, if we simply substitute the generators of $\bfU_\upsilon(\hatfraksl(n))$ with those of $\bfU_\upsilon(\fraksl(S^1))$ in Hayashi's formulas for the action of $\bfU_\upsilon(\hatfraksl(n))$ on the Fock space, we do not obtain an action of $\bfU_\upsilon(\fraksl(S^1))$ on $\calF_\R$. The way we construct the action of $\bfU_\upsilon(\fraksl(S^1))$ here consists of generalizing the geometric folding procedure due to Varagnolo and Vasserot which is based on the theory of Hall algebras. Precisely, the folding procedure is induced by an homomorphism between the Hall algebra of the continuum finite type A quiver $\R$ and the Hall algebra of the continuum affine type A quiver $S^1$, which realize a positive part of $\bfU_\upsilon(\fraksl(\R))$  and $\bfU_\upsilon(\fraksl(S^1))$, respectively. The explicit formulas of our folding procedure are quite involved (see~Theorem~\ref{thm:Fock-S1}): in particular, since $\bfU_\upsilon(\fraksl(S^1))$ is generated by elements $E_J,F_J,K_J$ for $J$ an interval in $S^1$, the absence of simple roots means that the folding procedure involves breaking up $J$ into finitely many smaller intervals in all possible ways. 
We obtain:
\begin{theorem*}[cf.\ Theorem~\ref{thm:Fock-S1}]
There exists a natural action of $\bfU_\upsilon(\fraksl(S^1))$ on $\calF_\R$, which strictly contains the irreducible highest weight representation generated by the vacuum vector $\vert 0 \rangle$.
\end{theorem*}

As opposed to the case of the Fock space of $\bfU_{\upsilon}(\hatfraksl(n))$ (see e.g. \cite{art:kashiwaramiwastern1995}), $\calF_{\R}$ contains no highest weight vector other than the vacuum vector $\vert 0 \rangle$. On the other hand, it is not a cyclic representation. It is, however, a cyclic representation of the Hall algebra of $S^1_\R$ (which strictly contains $\bfU^-_{\upsilon}(\fraksl(S^1))$). Its precise structure will be studied in a sequel to this paper.

\subsection{Higher level Fock spaces}

Above, we provide a continuum analog of the level one Fock space representation with fundamental weight $\Lambda_0$. One can take the $\ell$ tensor product $\calF_\R[\ell]$ of $\calF_\R$. It can be endowed with a well-defined action of $\bfU_\upsilon(\fraksl(\R))$. This action is defined via the coproduct: even though the coproduct is only topological and takes value in a suitable completion of tensor products of $\bfU_\upsilon(\fraksl(\R))$, we do not need to complete $\calF_\R[\ell]$ as well to get an action. In another direction, one can translate the pyramids along the $\R$-axis: this will give rise to (level $\ell\geq 1$) Fock space representations associated to other fundamental weights\footnote{We believe that, in our theory, weights should correspond to distribution in $\F(\R)$.}. Tensor products of Fock spaces with \textit{different} highest weights are more delicate to handle. We will address the study of these variations of the construction of the Fock space representation for the line and circle quantum groups in \cite{art:salaschiffmann2019}.

\subsection{$\K$-Variations}

Originally, it is a version of the quantum group associated to the \textit{rational} circle $S^1_\Q = \Q /\Z$ or of the rationals $\Q$ which appears in \cite{art:salaschiffmann2017}. One can define such quantum groups $\bfU_\upsilon(\fraksl(\K))$ and $\bfU_\upsilon(\fraksl(S^1_\K))$, for any subset of $\K \subset \R$ which is $1$-periodic, where $S^1_\K\coloneqq \K/\Z$ for $\Z\subsetneq \K$. When $\K=\Z$, our quantum group $\bfU_\upsilon(\fraksl(\Z))$ coincides with $\bfU_\upsilon(\fraksl(\infty))$, but it is much bigger when $\K=\Q$ or $\K=\R$. Likewise, for $\K=\frac{1}{n}\Z$ we have $\bfU_\upsilon(\fraksl(S^1_\K)) \simeq \bfU_\upsilon(\widehat{\fraksl}(n)))$ but is is much larger in the other two cases.
The Fock space construction can be carried out for any such $\K$, and we do it in this generality. For instance, the Fock space representation $\calF_\Q$ of $\bfU_\upsilon(\fraksl(S^1_\Q))$ has a basis indexed by $\Q$-pyramids, which are pyramids whose jumps occur only at rational numbers. When $\K=\Z$ we recover the usual Fock space action of $\bfU_\upsilon(\fraksl(\infty))$, up to some scaling factors.

\subsection{Line and circle quantum groups vs. continuum quantum groups}

In \cite{appel-sala-schiffmann-18}, the authors, together with Andrea Appel, define a \emph{continuum} generalization of the Kac--Moody Lie algebras, associated with a \emph{continuum} generalization of the notion of a quiver. In particular, the vertex set of a quiver is replaced by an Hausdorff topological space $X$, and the vertices of the quiver are replaced by connected intervals in $X$ (we refer to \emph{loc. cit.} for the relevant definitions). We denote by $\frakg_X$ the corresponding continuum Kac--Moody Lie algebra. It is proved in \emph{loc. cit.} that  $\frakg_\K$ coincides with $\fraksl(\K)$, while the Lie algebra $\frakg_{S^1_\K}$ differs from $\fraksl(S^1_\K)$ by an \emph{Heisenberg Lie algebra of order one}. In \cite{appel-sala-19}, the first-named author, together with Andrea Appel, constructed a topological Lie bialgebra structure on the continuum Kac--Moody Lie algebra $\frakg_X$ of any space $X$ and defined algebraically the quantization $\bfU_\upsilon (\frakg_X)$ of $\frakg_X$, which is called the \emph{continuum quantum group} of $X$. It is proved in \emph{loc. cit.} that $\bfU_\upsilon (\frakg_\K)$ coincides with $\bfU_\upsilon(\fraksl(\K))$, while the quantum group $\bfU_\upsilon (\frakg_{S^1_\K})$ differs from $\bfU_\upsilon(\fraksl(S^1_\K))$ by a \emph{quantum Heisenberg algebra of order one}. Finally, in \cite{appel-kuwagaki-sala-schiffmann-18}, we are constructing the quantum group $\bfU_\upsilon (\frakg_X)$ via the theory of Hall algebras. We expect to extend some of the techniques of this paper to such a more general setting, in particular the construction of a Fock space representation for any $\bfU_\upsilon (\frakg_X)$ will be the subject of a sequel to this paper.
%

\subsection{Outline}

The paper is organized as follows. In Section~\ref{s:partitionpyramids}, we recall the notions of partitions and Young diagrams, and describe their topological refinement as rational and real pyramids. Section~3 serves as a reminder about Hayashi's Fock space and Varagnolo-Vasserot's geometric interpretation of the folding procedure. In Section~\ref{s:Fockspaceline}, we introduce our Fock space $\calF_\K$ and the action of the quantum group $\bfU_\upsilon(\fraksl(\K))$, for $\K\in \{\Z, \Q, \R\}$, on it, while in Section~\ref{s:Fockspacecircle}, we define the action of $\bfU_\upsilon(\fraksl(S^1_\K))$, for $S^1_\K\coloneqq \K/\Z$ and $\K\in\{\Q, \R\}$, on $\calF_\K$. The proof that such an action is well-defined is given in Section~\ref{s:Fockspacecircle} and it is based on a variant of the ``folding procedure" depending on the theory of Hall algebras associated with certain topological quivers, as introduced in Section~\ref{s:Hallalgebra}.

\subsection*{Acknowledgements}

We thank Andrea Appel and Tatsuki Kuwagaki for helpful discussions and comments. We also thank the anonymous referees for useful suggestions and comments.

\subsection*{Notation and convention}

For any integer $n$, set
\begin{align}
	[n]\coloneqq \frac{\upsilon^n-\upsilon^{-n}}{\upsilon-\upsilon^{-1}}\quad\text{and}\quad [n]!\coloneqq [n]\, [n-1]\, \cdots \, [1]\ .
\end{align}
Set $\widetilde{\Q} \coloneqq\Q[\upsilon, \upsilon^{-1}]$ and put $q=\upsilon^2$. 

\bigskip\section{Partitions and pyramids}\label{s:partitionpyramids}

In this section, we introduce integral, rational, and real \emph{pyramids} and establish their basic properties. Rational and real pyramids are some ``continuous" generalization of the notion of partition, while we shall show that integral pyramids coincide with partitions.

\subsection{Recollection on partitions}\label{ss:partitions}

A \emph{partition} of a positive integer $n$ is a nonincreasing sequence of positive numbers $\lambda=(\lambda_1, \lambda_2, \cdots, \lambda_\ell)$ such that $\lambda_i\geq \lambda_{i+1}$ for $i=1, \ldots, \ell-1$ and $\vert\lambda\vert\coloneqq\sum_{a=1}^\ell \, \lambda_a=n.$ We call $\ell=\ell(\lambda)$ the \emph{length} of the partition $\lambda.$ Another description of a partition $\lambda$ of $n$ uses the notation $\lambda=(1^{m_1}\, 2^{m_2}\, \cdots)$, where $m_i=\#\{a\in\N\,\vert \, \lambda_a =i\}$ with $\sum_i \, i\, m_i=n$ and $\sum_i\, m_i=\ell(\lambda)$. We denote by $\Pi(n)$ the set of all partitions of $n$, and $\Pi\coloneqq \bigcup_n\, \Pi(n)$.  On the set $\Pi$ of all partitions there is a natural partial ordering called \emph{dominance ordering}: for two partitions $\mu$ and $\lambda$, we write $\mu\leq\lambda$ if and only if $\vert \mu\vert=\vert\lambda\vert$ and $\mu_1+\cdots+\mu_a\leq\lambda_1+\cdots+\lambda_a$ for all $a\geq1$. We write $\mu<\lambda$ if and only if $\mu\leq\lambda$ and $\mu\neq\lambda.$ 

One can associate with a partition $\lambda$ its \emph{Young diagram}, which is the set $Y_\lambda=\{(x,y)\in\Z_{>0}^2\,\vert\, 1\leq y \leq \ell(\lambda)\, ,\, 1\leq x \leq \lambda_y\}$. Then $\lambda_y$ is the length of the $y$-th row of $Y_\lambda$; we write $\vert Y_\lambda\vert=\vert\lambda\vert$ for the \emph{weight} of the Young diagram $Y_\lambda$. We shall identify a partition $\lambda$ with its Young diagram $Y_\lambda$. For example, with the partition $\lambda=(5,4,4,3,1,1)$ we associate the Young diagram $Y_\lambda$:
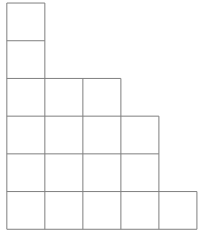
	\begin{align}
	\begin{tikzpicture}
	\draw[gray,very thin](-6,0) -- (-3.5,0);
	\draw[gray,very thin](-6,0.5) -- (-3.5,0.5);
	\draw[gray,very thin](-6,1) -- (-4,1);
	\draw[gray,very thin](-6,1.5) -- (-4,1.5);
	\draw[gray,very thin](-6,2) -- (-4.5,2);
	\draw[gray,very thin](-6,2.5) -- (-5.5,2.5);
	\draw[gray,very thin](-6,3) -- (-5.5,3);
	\draw[gray,very thin](-6,0) -- (-6,3);
	\draw[gray,very thin](-5.5,0) -- (-5.5,3);
	\draw[gray,very thin](-5,0) -- (-5,2);
	\draw[gray,very thin](-4.5,0) -- (-4.5,2);
	\draw[gray,very thin](-4,0) -- (-4,1.5);
	\draw[gray,very thin](-3.5,0) -- (-3.5,0.5);
	\end{tikzpicture}
	\end{align}
For a partition $\lambda$, the \emph{transpose partition} $\lambda'$ is the partition whose Young tableau is $Y_{\lambda'}\coloneqq\{(b,a)\in \Z_{>0}^2\,\vert\,(a,b)\in Y_\lambda\}$. 

Finally, we call \emph{standard content} of $s=(x,y)\in Y_\lambda$ the quantity $c(s)\coloneqq x-y$. We say that box $s$ is \emph{of color $i$} if $c(s)=i$. An \emph{addable $i$-box} is a box of color $i$ which can be added to $Y_\lambda$ in such a way that the new diagram still comes from a partition, similarly a \emph{removable $i$-box} is a box of color $i$ which can be removed from $Y_\lambda$. For $i\in \Z_{> 0}$, define
\begin{align}
	n_i(\lambda)\coloneqq \# \{\text{addable $i$-boxes of $Y_\lambda$}\}-\# \{\text{removable $i$-boxes of $Y_\lambda$}\}\ .
\end{align}
\begin{remark}
	Note that
	\begin{align}
		n_i(\lambda)=
		\begin{cases}
			1 & \text{if there exists an addable box of color $i$}\ , \\
			-1 & \text{if there exists a removable box of color $i$} \ ,\\
			0 & \text{otherwise} \ .
		\end{cases}
	\end{align}
\end{remark}

\subsection{Integral pyramids}

We now provide another combinatorial realization of partitions very similar to that of Maya diagrams\footnote{A \emph{Maya diagram} is a sequence $\{m(n)\}_{n\in \Z}$ which consists of 0 or 1 and satisfies the following property: there exist $N, M \in \Z$ such that for all $n > N$ (resp.\ $n < M$), $m(n) = 1$ (resp.\ $m(n) = 0$).}, which we call ``integral pyramids". Below, 'increasing' and 'decreasing' are meant in the broad sense.

\begin{definition}
An \emph{integral pyramid} is a function $p\colon \Z \to \N$ satisfying the following properties:
\begin{enumerate}\itemsep0.2cm
	\item[i)] $p(n)=0$ for $\vert n\vert \gg 0$;
	\item[ii)] $p(0) =\max \{p(n)\,\vert\, n \in \Z\}$;
	\item[iii)] $p$ is increasing on $\Z_{-}$ and decreasing on $\Z_+$;
	\item[iv)] $\vert p(n)-p(n+1)\vert \leq 1$ for all $n \in \Z$.
\end{enumerate}
\end{definition}

We may represent a pyramid as a box diagram in which we draw $p(n)$ boxes over the integer $n$: for instance, we represent the pyramid such that $p(-5)=p(-4)=1,p(-3)=2, p(-2)=p(-1)=3,  p(0)=4,  p(1)=3,  p(2)=p(3)=2,  p(4)=p(5)=p(6)=1$ and $p(n)=0$ for $n < -5$ or $n >6$ as follows:
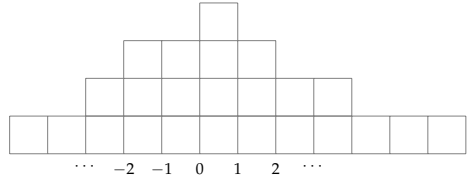
\begin{align}
	\begin{tikzpicture}
	\draw[step=0.5cm,gray,very thin] (-0.5,0) grid (5.5,0.5);
	\draw[step=0.5cm,gray,very thin] (0.5,0.5) grid (4,1);
	\draw[step=0.5cm,gray,very thin] (1,1) grid (3,1.5);
	\draw[gray,very thin] (1,1) -- (1,1.5);
	\draw[gray,very thin] (2,1.5) -- (2,2);
	\draw[step=0.5cm,gray,very thin] (2,1.5) grid (2.5,2);
	\draw (2,0) circle (0) node[anchor=north] {\tiny{$0$}};
	\draw (2.5,0) circle (0) node[anchor=north] {\tiny{$1$}};
	\draw (1.5,0) circle (0) node[anchor=north] {\tiny{$-1$}};
	\draw (3,0) circle (0) node[anchor=north] {\tiny{$2$}};
	\draw (1,0) circle (0) node[anchor=north] {\tiny{$-2$}};
	\draw (3.5,0) circle (0) node[anchor=north] {\tiny{$\cdots$}};
	\draw (0.5,0) circle (0) node[anchor=north] {\tiny{$\cdots$}};
	\end{tikzpicture}
\end{align}

The set of pyramids with $n$ boxes will be denoted $\mathsf{Pyr}(n)$.
\begin{lemma}\label{lem:bijection}
	There is a canonical bijection between $\mathsf{Pyr}(n)$ and the set $\Pi(n)$ of partitions of $n$.
\end{lemma}
\begin{proof}
	Let us define a map $f\colon \Pi(n)\to \mathsf{Pyr}(n)$ as follows. Let $\lambda$ be a partition and let $Y_\lambda$ be its Young diagram. We fill the boxes of $Y_{\lambda}$ with the standard content $c(s)=y-x$ if $s=(x,y)$. Then $f(\lambda)$ is the pyramid $p_{\lambda}$ defined by $p_{\lambda}(n)\coloneqq\#\{ s \in Y_\lambda\,\vert\, c(s)=n\}$ for any integer $n$. Viceversa, we can define a map $g\colon \mathsf{Pyr}(n)\to \Pi(n)$ by assigning to a pyramid $p$ the unique Young diagram for which the number of boxes $s$ for which $c(s)=n$ is exactly $p(n)$ for any $n\in \Z$. It is straighforward to see that $g$ is the inverse of $f$.
\end{proof}

\begin{ex}
	For instance, to the partition $\lambda=(5,4,4,3,1,1)$ we associate the pyramid $p_{\lambda}(-5)=p_{\lambda}(-4)=p_{\lambda}(-3)=1, p_{\lambda}(-2)=2, p_{\lambda}(-1)= p_{\lambda}(0)=p_{\lambda}(1)=3, p_{\lambda}(2)=2, p_{\lambda}(3)=p_{\lambda}(4)=1$.
	
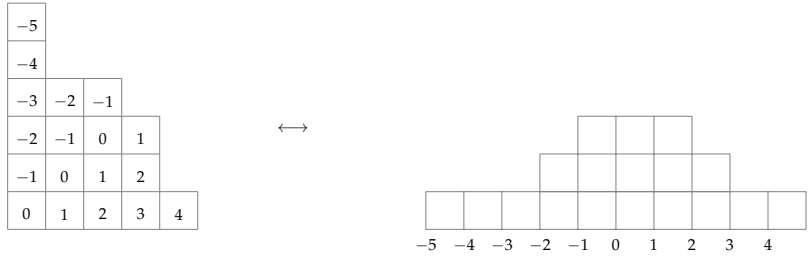
\begin{align}
		\begin{tikzpicture}
		\draw[gray,very thin](-6,0) -- (-3.5,0);
		\draw[gray,very thin](-6,0.5) -- (-3.5,0.5);
		\draw[gray,very thin](-6,1) -- (-4,1);
		\draw[gray,very thin](-6,1.5) -- (-4,1.5);
		\draw[gray,very thin](-6,2) -- (-4.5,2);
		\draw[gray,very thin](-6,2.5) -- (-5.5,2.5);
		\draw[gray,very thin](-6,3) -- (-5.5,3);
		\draw[gray,very thin](-6,0) -- (-6,3);
		\draw[gray,very thin](-5.5,0) -- (-5.5,3);
		\draw[gray,very thin](-5,0) -- (-5,2);
		\draw[gray,very thin](-4.5,0) -- (-4.5,2);
		\draw[gray,very thin](-4,0) -- (-4,1.5);
		\draw[gray,very thin](-3.5,0) -- (-3.5,0.5);
		\draw[step=0.5cm,gray,very thin] (-0.5,0) grid (4.5,0.5);
		\draw[step=0.5cm,gray,very thin] (1,0.5) grid (3.5,1);
		\draw[step=0.5cm,gray,very thin] (1.5,1) grid (3,1.5);
		\draw[gray,very thin] (1,0.5) -- (1,1);
		\draw[gray,very thin] (1.5,1) -- (1.5,1.5);
		\draw (2,0) circle (0) node[anchor=north] {\tiny{$0$}};
		\draw (2.5,0) circle (0) node[anchor=north] {\tiny{$1$}};
		\draw (1.5,0) circle (0) node[anchor=north] {\tiny{$-1$}};
		\draw (3,0) circle (0) node[anchor=north] {\tiny{$2$}};
		\draw (1,0) circle (0) node[anchor=north] {\tiny{$-2$}};
		\draw (3.5,0) circle (0) node[anchor=north] {\tiny{$3$}};
		\draw (0.5,0) circle (0) node[anchor=north] {\tiny{$-3$}};
		\draw (4,0) circle (0) node[anchor=north] {\tiny{$4$}};
		\draw (0,0) circle (0) node[anchor=north] {\tiny{$-4$}};
		\draw (-0.5,0) circle (0) node[anchor=north] {\tiny{$-5$}};
		\draw (-2.25,1.5) circle (0) node[anchor=north] {\tiny{$\longleftrightarrow$}};
		\draw (-5.75,0.4) circle (0) node[anchor=north] {\tiny{$0$}};
		\draw (-5.25,0.9) circle (0) node[anchor=north] {\tiny{$0$}};
		\draw (-4.75,1.4) circle (0) node[anchor=north] {\tiny{$0$}};
		\draw (-5.25,0.4) circle (0) node[anchor=north] {\tiny{$1$}};
		\draw (-4.75,0.9) circle (0) node[anchor=north] {\tiny{$1$}};
		\draw (-4.25,1.4) circle (0) node[anchor=north] {\tiny{$1$}};
		\draw (-5.75,0.9) circle (0) node[anchor=north] {\tiny{$-1$}};
		\draw (-5.25,1.4) circle (0) node[anchor=north] {\tiny{$-1$}};
		\draw (-4.75,1.9) circle (0) node[anchor=north] {\tiny{$-1$}};
		\draw (-5.75,1.4) circle (0) node[anchor=north] {\tiny{$-2$}};
		\draw (-5.75,1.9) circle (0) node[anchor=north] {\tiny{$-3$}};
		\draw (-5.75,2.4) circle (0) node[anchor=north] {\tiny{$-4$}};
		\draw (-5.75,2.9) circle (0) node[anchor=north] {\tiny{$-5$}};
		\draw (-5.25,1.9) circle (0) node[anchor=north] {\tiny{$-2$}};
		\draw (-4.75,0.4) circle (0) node[anchor=north] {\tiny{$2$}};
		\draw (-4.25,0.4) circle (0) node[anchor=north] {\tiny{$3$}};
		\draw (-3.75,0.4) circle (0) node[anchor=north] {\tiny{$4$}};
		\draw (-4.25,0.9) circle (0) node[anchor=north] {\tiny{$2$}};
		\end{tikzpicture}
	\end{align}
	Pictorially, the bijection amounts to folding clockwise by $90$ degrees the part of the pyramid which lies over $-\N$ (and shifting accordingly each layer of the pyramid to the right by a number of boxes equal to its height). Conversely, starting from a partition, one may tilt its Young diagram 45 degrees to the left (so that it stands on its corner) and let gravity act to obtain the associated pyramid.\hfill$\triangle$
\end{ex}

It is easy to translate several classical notions from partitions to pyramids. For instance,
\begin{align}\label{eq:lengthpyr}
	\ell(\lambda)&=\sup\,\{n\,\vert\, p_{\lambda}(-n) \neq 0\}\ ,\\
	\label{eq:transposepyr}
	p_{\lambda'}&=p_{\lambda} \circ \iota \ ,
\end{align}
where $\lambda'$ is the transpose partition of $\lambda$ and $\iota \colon \Z \to \Z$ sends $n$ to $-n$, and
\begin{align}\label{eq:sizepyr}
	\vert \lambda \vert = \int p_{\lambda}\coloneqq\sum_n p(n)\ .
\end{align}
The dominance ordering is less pleasant. It translates into the following set of inequalities:
\begin{align}\label{eq:dominancepyr}
	p \geq q \qquad \Leftrightarrow \qquad \int \inf (p, \kappa_n) \geq \int \inf (q,\kappa_n) \quad \text{for all $n$} \ ,
\end{align}
where 
\begin{align}
	\kappa_n(i)\coloneqq
	\begin{cases}
		n & \text{if } i \geq 0\ ,\\ 
		n+i & \text{if }  -n < i <0\ ,\\ 
		0 & \text{if }  i \leq -n\ .
	\end{cases}
\end{align}

\subsection{Rational and real pyramids}

We now generalize the concept of pyramids to allow for ``jumps" of the function $p$ located at non-integral points. More precisely, let us introduce the following definition.
\begin{definition}
A \emph{rational} (resp.\ \emph{real}) \emph{pyramid} is defined to be a function $p\colon \Q \to \N$ (resp.\ a function $p\colon \R \to \N$) satisfying the following properties (we set $\K=\Q, \R$ accordingly):
\begin{enumerate}\itemsep0.2cm
	\item[i)] $p(x)=0$ for $\vert x\vert \gg 0$;
	\item[ii)] $p(0) =\max \{p(x)\,\vert \, x \in \K\}$;
	\item[iii)] $p$ is increasing on $\K_{-}$ and decreasing on $\K_+$;
	\item[iv)] $p$ is right--continuous and piecewise constant, with finitely many points of discontinuity;
	\item[v)] $\vert p_+(x)-p_{-}(x)\vert  \leq 1$ for all $x \in \K$.
\end{enumerate}
\end{definition}
We extend the notions of length, transpose and size to rational or real pyramids using \eqref{eq:lengthpyr}, \eqref{eq:transposepyr}, \eqref{eq:sizepyr} respectively. The notion of dominance is also extended accordingly, where we now require the inequality \eqref{eq:dominancepyr} for any $n \in \K$. 

We may still represent rational or real pyramids as diagrams (now simply the graph of the function $p$). On the other hand, it is not possible to represent rational or real pyramids as Young diagrams anymore, as the operation of ``folding clockwise by $90$ degrees" does not make sense. We denote by $\mathsf{Pyr}(\K)$ the set of all $\K$-pyramids. We will sometimes also denote by $\mathsf{Pyr}(\K)(u)$ the set of $\K$-pyramids with size $u$.

\subsection{Addable and removable intervals}

Fix $\K \in \{\Z, \Q, \R\}$.
\begin{definition}
	We call a  \emph{$\K$-interval} a closed-open interval of the form $J=[a,b[\coloneqq \{x\in \R\;\vert\; a\leq x<b\}$ with $a,b\in \K$. We denote by $\cf_J$ the \emph{characteristic function} of $J$.
\end{definition}

\begin{definition}
	Let $p$ be a $\K$-pyramid. A $\K$-interval $J$ is called \emph{addable} (resp.\ \emph{removable}) if $p+\cf_J$ (resp.\ $p-\cf_J$) is still a $\K$-pyramid. 
\end{definition}

\begin{remark}
	When $\K=\Z$, as proved in Lemma~\ref{lem:bijection} a pyramid $p$ corresponds to a partition $\lambda_p$. A length one $\Z$-interval $J$ gives rise to a box $s\coloneqq (1+p(a), 1+p(a)-a)$. Hence $J$ is addable (resp.\ removable) if $s$ is addable (resp.\ removable) in the sense defined in Section~\ref{ss:partitions}. More generally, when $J$ is of arbitrary length, it corresponds to a connected strip.
\end{remark}

\begin{example}
Here is an example of an addable $\Z$-interval and the corresponding strip:
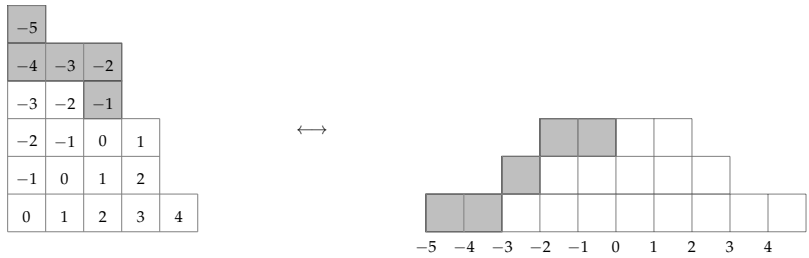
\begin{align}
\begin{tikzpicture}
\draw [fill=lightgray,very thin] (-6,2.5) rectangle (-5.5,3);
\draw [fill=lightgray,very thin] (-6,2) rectangle (-4.5,2.5);
\draw [fill=lightgray,very thin] (-5,1.5) rectangle (-4.5,2);
\draw [fill=lightgray,very thin] (-0.5,0) rectangle (0.5,0.5);
\draw [fill=lightgray,very thin] (0.5,0.5) rectangle (1,1);
\draw [fill=lightgray,very thin] (1,1) rectangle (2,1.5);
\draw[gray,very thin](-6,0) -- (-3.5,0);
\draw[gray,very thin](-6,0.5) -- (-3.5,0.5);
\draw[gray,very thin](-6,1) -- (-4,1);
\draw[gray,very thin](-6,1.5) -- (-4,1.5);
\draw[gray,very thin](-6,2) -- (-4.5,2);
\draw[gray,very thin](-6,2.5) -- (-4.5,2.5);
\draw[gray,very thin](-6,3) -- (-5.5,3);
\draw[gray,very thin](-6,0) -- (-6,3);
\draw[gray,very thin](-5.5,0) -- (-5.5,3);
\draw[gray,very thin](-5,0) -- (-5,2.5);
\draw[gray,very thin](-4.5,0) -- (-4.5,2.5);
\draw[gray,very thin](-4,0) -- (-4,1.5);
\draw[gray,very thin](-3.5,0) -- (-3.5,0.5);
\draw[step=0.5cm,gray,very thin] (-0.5,0) grid (4.5,0.5);
\draw[step=0.5cm,gray,very thin] (0.5,0.5) grid (3.5,1);
\draw[step=0.5cm,gray,very thin] (1,1) grid (3,1.5);
\draw[gray,very thin] (1,0.5) -- (1,1);
\draw[gray,very thin] (1,1) -- (1,1.5);
\draw[gray,very thin] (1.5,1) -- (1.5,1.5);
\draw (2,0) circle (0) node[anchor=north] {\tiny{$0$}};
\draw (2.5,0) circle (0) node[anchor=north] {\tiny{$1$}};
\draw (-5.25,2.4) circle (0) node[anchor=north] {\tiny{$-3$}};
\draw (-4.75,2.4) circle (0) node[anchor=north] {\tiny{$-2$}};
\draw (1.5,0) circle (0) node[anchor=north] {\tiny{$-1$}};
\draw (3,0) circle (0) node[anchor=north] {\tiny{$2$}};
\draw (1,0) circle (0) node[anchor=north] {\tiny{$-2$}};
\draw (3.5,0) circle (0) node[anchor=north] {\tiny{$3$}};
\draw (0.5,0) circle (0) node[anchor=north] {\tiny{$-3$}};
\draw (4,0) circle (0) node[anchor=north] {\tiny{$4$}};
\draw (0,0) circle (0) node[anchor=north] {\tiny{$-4$}};
\draw (-0.5,0) circle (0) node[anchor=north] {\tiny{$-5$}};
\draw (-2,1.5) circle (0) node[anchor=north] {\tiny{$\longleftrightarrow$}};
\draw (-5.75,0.4) circle (0) node[anchor=north] {\tiny{$0$}};
\draw (-5.25,0.9) circle (0) node[anchor=north] {\tiny{$0$}};
\draw (-4.75,1.4) circle (0) node[anchor=north] {\tiny{$0$}};
\draw (-5.25,0.4) circle (0) node[anchor=north] {\tiny{$1$}};
\draw (-4.75,0.9) circle (0) node[anchor=north] {\tiny{$1$}};
\draw (-4.25,1.4) circle (0) node[anchor=north] {\tiny{$1$}};
\draw (-5.75,0.9) circle (0) node[anchor=north] {\tiny{$-1$}};
\draw (-5.25,1.4) circle (0) node[anchor=north] {\tiny{$-1$}};
\draw (-4.75,1.9) circle (0) node[anchor=north] {\tiny{$-1$}};
\draw (-5.75,1.4) circle (0) node[anchor=north] {\tiny{$-2$}};
\draw (-5.75,1.9) circle (0) node[anchor=north] {\tiny{$-3$}};
\draw (-5.75,2.4) circle (0) node[anchor=north] {\tiny{$-4$}};
\draw (-5.75,2.9) circle (0) node[anchor=north] {\tiny{$-5$}};
\draw (-5.25,1.9) circle (0) node[anchor=north] {\tiny{$-2$}};
\draw (-4.75,0.4) circle (0) node[anchor=north] {\tiny{$2$}};
\draw (-4.25,0.4) circle (0) node[anchor=north] {\tiny{$3$}};
\draw (-3.75,0.4) circle (0) node[anchor=north] {\tiny{$4$}};
\draw (-4.25,0.9) circle (0) node[anchor=north] {\tiny{$2$}};
\end{tikzpicture}
\end{align}

and here is an example of an addable $\Q$ (or $\R$)-interval (in this case, $J=[-\frac{11}{5}, \frac{2}{5}[$):
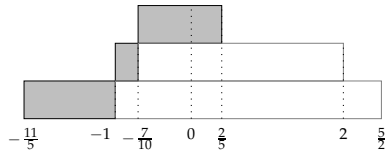
\begin{align}
\begin{tikzpicture}
\draw [fill=lightgray,very thin] (-0.5,0) rectangle (0.7,0.5);
\draw [fill=lightgray,very thin] (0.7,0.5) rectangle (1,1);
\draw [gray,very thin] (0.7,0) rectangle (4.2,0.5);
\draw [fill=lightgray,very thin] (1,1) rectangle (2.1,1.5);
\draw [gray,very thin] (1,0.5) rectangle (3.7,1);
\draw[dotted] (1.7,0) -- (1.7,1.5);
\draw[dotted] (0.7,0) -- (0.7,0.5);
\draw[dotted] (1,0) -- (1,1);
\draw[dotted] (2.1,0) -- (2.1,1.5);
\draw[dotted] (3.7,0) -- (3.7,1);
\draw (1.7,0) circle (0) node[anchor=north] {\tiny{$0$}};
\draw (-0.5,0) circle (0) node[anchor=north] {\tiny{$-\frac{11}{5}$}};
\draw (0.5,0) circle (0) node[anchor=north] {\tiny{$-1$}};
\draw (1,0) circle (0) node[anchor=north] {\tiny{$-\frac{7}{10}$}};
\draw (2.1,0) circle (0) node[anchor=north] {\tiny{$\frac{2}{5}$}};
\draw (3.7,0) circle (0) node[anchor=north] {\tiny{$2$}};
\draw (4.2,0) circle (0) node[anchor=north] {\tiny{$\frac{5}{2}$}};
\end{tikzpicture}
\end{align}
Note that we could not have added $[-\frac{11}{5},\frac{5}{2}[$ since there would be a jump of $2$ over $\frac{5}{2}$, nor could we have added the interval $[-\frac{11}{5}, 0[$ since $0$ would not be the maximum anymore.
\end{example}

\begin{definition}
	Let $p$ be a $\K$-pyramid. We put
	\begin{align}
		D_\K(p)\coloneqq \{y \in \K\, \vert\, y \text{ is a point of discontinuity of } p\}\ .
	\end{align}
and call this the \emph{set of discontinuities of $p$}.
\end{definition}

\begin{remark}\label{rem:action-nonzero}
	Let $J=[a,b[$ be a $\K$-interval and let $p$ be a $\K$-pyramid. Then $p+\cf_J$ is a pyramid if and only if one of the following mutually exclusive cases occurs
	\begin{itemize}\itemsep0.2cm
		\item $a\in D_\K(p)$ if $0<a$;
		\item $a,b \notin D_\K(p)$ if $a\leq 0 <b$;
		\item $b\in D_\K(p)$ if $b<0$.
	\end{itemize}
	Similarly, $p-\cf_J$ is a pyramid if and only if one the following mutually exclusive cases occurs
	\begin{itemize}\itemsep0.2cm
		\item $b\in D_\K(p)$ if $0<a$;
		\item $a,b \in D_\K(p)$ if $a\leq 0 <b$;
		\item $a\in D_\K(p)$ if $b<0$.
	\end{itemize}
\end{remark}

Let $p$ be a $\K$-pyramid and let $J$ be an addable interval. We define the \emph{$p$-height} of $J$ as 
\begin{align}
	\mathsf{ht}_p(J)\coloneqq \sup\,\{\vert p(x)-p(y)\vert \,\vert \, x,y \in J\}\ .
\end{align}
We will also consider the variants 
\begin{align}
	\mathsf{ht}^{\pm}_p(J)\coloneqq \sup\,\{\vert p(x)-p(y)\vert \,\vert\, x,y \in J \cap \K_{\pm}\}\ ,
\end{align}
so that $\mathsf{ht}_p(J)=\max\{ \mathsf{ht}^+_p(J), \mathsf{ht}^-_p(J)\}$. 

\begin{example}
In the last example above, we have, for $J=[-\frac{11}{5},\frac{5}{2}[$,
\begin{align}
	\mathsf{ht}_p(J)=2\ , \quad \mathsf{ht}^-_p(J)=2\ , \quad \mathsf{ht}^+_p(J)=0 \ .
\end{align}
\end{example}

\bigskip\section{Recollections on the standard Fock space representation}\label{s:Fockspace}

\subsection{Type A quantum groups }

Let $\bfU_\upsilon(\fraksl(\infty))$ (resp.\ $\bfU_\upsilon(\hatfraksl(n))$) be the \emph{quantized enveloping algebra of $\fraksl(\infty)$} (resp.\ \emph{of $\hatfraksl(n)$}), i.e., the unital associative $\widetilde{\Q}$-Hopf algebra generated by $E_i, F_i, K_i^{\pm 1}$, for $i\in \Z$ (resp.\ $i\in \Z/n\Z$), subject to the Drinfeld-Jimbo type relations
\begin{align}
K_i\, K_i^{-1}=1=K_i^{-1}\, K_i\ , \quad K_i\, K_j=K_j\, K_i \ , 
\end{align}
\begin{align}
K_i E_j K_i^{-1}= \upsilon^{a_{i,j}} E_j\ , \quad K_i F_j K_i^{-1}= \upsilon^{-a_{i,j}} F_j\ , \quad [E_i, F_j]=\delta_{i,j}\frac{K_i-K_i^{-1}}{\upsilon-\upsilon^{-1}}\ ,
\end{align}
\begin{align}
\sum_{k=0}^{1-a_{i,j}}\, (-1)^k E_i^{(k)}\, E_j E_i^{(1-a_{i,j}-k)}=0=\sum_{k=0}^{1-a_{i,j}}\, (-1)^k F_i^{(k)}\, F_j F_i^{(1-a_{i,j}-k)}\quad\text{if $i\neq j$} \ ,
\end{align}
where\footnote{For $i, j\in \Z/n\Z$, we interpret $\vert i-j\vert$ modulo $n$.}
\begin{align}
a_{i,j}\coloneqq\begin{cases}
2 & \text{if $i=j$}\ ,\\
-1 & \text{if $\vert i-j\vert =1$}\ , \\
0 & \text{otherwise} \ ,
\end{cases}
\end{align}
and
\begin{align}
E_i^{(k)}\coloneqq \frac{E_i^k}{[k]!}\quad\text{and}\quad F_i^{(k)}\coloneqq \frac{F_i^k}{[k]!}\ .
\end{align}
The coproduct is
\begin{align}
\Delta(K_i)=K_i\otimes K_i\ , \quad \Delta(E_i)=E_i\otimes 1+K_i\otimes E_i\ , \quad \Delta(F_i)=F_i\otimes K_i^{-1}+1\otimes F_i \ .
\end{align}

\subsection{Fock spaces for type A quantum groups}

Set
\begin{align}
\calF_\Z\coloneqq  \bigoplus_{\lambda\in\mathsf{Pyr}(\Z)}\,  \widetilde \Q \vert \lambda\rangle\ .
\end{align}
The action of $\bfU_\upsilon(\fraksl(\infty))$ on $\calF_\Z$ given by
\begin{align}
E_i\vert \lambda \rangle&\coloneqq \begin{cases}
\vert \nu \rangle & \text{if $Y_\lambda\setminus Y_\nu$ is a box with color $i$}\ , \\[4pt]
0 & \text{otherwise} \ ,
\end{cases}\\[6pt]
F_i\vert \lambda \rangle&\coloneqq\begin{cases}
\vert \mu \rangle & \text{if $Y_\mu\setminus Y_\lambda$ is a box with color $i$} \ ,\\[4pt]
0 & \text{otherwise}\ ,
\end{cases}\\[6pt]
K_i\vert \lambda \rangle&\coloneqq\upsilon^{n_i(\lambda)}\vert \lambda \rangle
\end{align}
for $i\in \Z$. 

Following \cite{art:hayashi1990, art:misramiwa1990}, the action of $\bfU_\upsilon(\hatfraksl(n))$ on $\calF_\Z$ is defined via the ``folding procedure'' as
\begin{align}
E_{\bar{\imath}}\vert \lambda \rangle&\coloneqq \sum_{j\in \Z\,:\, j=\bar{\imath}\bmod{n}}\, \upsilon^{-n_j^{-}(\lambda)}\, E_j \vert \lambda \rangle\ , \\[6pt]
F_{\bar{\imath}}\vert \lambda \rangle&\coloneqq  \sum_{j\in \Z\,:\, j=\bar{\imath}\bmod{n}}\, \upsilon^{n_j^{+}(\lambda)}\, F_j \vert \lambda \rangle\ , \\[6pt]
K_{\bar{\imath}}\vert \lambda \rangle&\coloneqq \upsilon^{n_{\bar{\imath}}(\lambda)}\vert \lambda \rangle
\end{align}
for $\bar{\imath}\in \Z/n\Z$. Here
\begin{align}
	n_{\bar{\imath}}(\lambda)\coloneqq \sum_{j=\bar{\imath} \bmod{n}}\, n_j(\lambda)\ , \quad n_j^{-}(\lambda)\coloneqq \sum_{k<j\,:\, k=j\bmod{n}}\, n_k(\lambda)\ ,\quad n_j^{+}(\lambda)\coloneqq \sum_{k>j\,:\, k=j\bmod{n}}\, n_k(\lambda) \ .
\end{align}

\subsection{Varagnolo-Vasserot construction}\label{ss:VVbusiness}

For future use, we now give, following Varagnolo and Vasserot \cite{art:varagnolovasserot1999}, a geometric interpretation of the above formulas. This uses the notions of Hall algebras and quiver representations which we first briefly recall (see \cite{art:schiffmann-lectures} for details). 

A quiver is an oriented graph $\calQ=(I,H)$ where $I$ stands for the set of vertices and $H$ for the set of oriented edges; for $h \in H$, we denote by $h'$, resp. $h''$ its source and target vertices. A representation of $\calQ$ over a field $k$ is by definition the data of an $I$-graded finite-dimensional $k$-vector space $V=\bigoplus_i V_i$ together with a collection of linear maps $x_h \in \mathsf{Hom}(V_{h'},V_{h''})$ for $h \in H$. The dimension of a representation $M=(V,(x_h)_h)$ is defined as $\underline{\dim}(M)=(\dim(V_i))_i \in \N^I$. Representations of a quiver $\calQ$ over a field $k$ form an abelian category $\Rep_k\calQ$, which is known to be \emph{hereditary}, which means that $\mathsf{Ext}^i(M,N)=\{0\}$ for any pair of representations $(M,N)$ and any $i >1$.  The Euler form 
\begin{align}
	\langle\cdot ,\cdot\rangle\colon \sfK_0(\Rep_k\calQ)\otimes  \sfK_0(\Rep_k\calQ) \to \Z\ ,\quad  M\otimes N \mapsto \dim(\mathsf{Hom}(M,N))- \dim(\mathsf{Ext}^1(M,N))
\end{align}
factors through the dimension map and is explicitly given by
\begin{align}
	\langle\cdot ,\cdot\rangle\colon \Z^I \otimes \Z^I \to \Z\ , \quad \langle d,d'\rangle=\sum_i d_i d'_i - \sum_h d_{h'}d'_{h''}\ .
\end{align}
To every vertex $i$ there corresponds a simple one-dimensional representation $S_i$, supported at the vertex $i$ and with all linear zero maps $x_h$ equal to zero.

Another point of view, better suited for our purposes, involves the path algebra $k\calQ$: this is the algebra with basis given by the set of oriented paths and multiplication given by concatenation of paths (whenever possible, zero otherwise). The algebra $k\calQ$ has idempotents $e_i$, for $i \in I$, corresponding to length zero paths starting and ending at a vertex $i$. One can easily check that $k\calQ$ is generated by the collection of idempotents $e_i, i\in I$ and the length one paths $h \in H$, and that the assignment $M \mapsto (\bigoplus_i e_i M, \rho_M(h)_h)$ defines an equivalence between the category of finite-dimensional $k\calQ$-modules and the category of representations of $\calQ$ over $k$. Here $\rho_M\colon k\calQ\to \mathsf{End}(M)$ is the map defining the $k\calQ$-module structure of $M$.

From now on, we fix $k$ to be a finite field $\mathbb{F}_q$ and set $\upsilon\coloneqq q^{1/2}$. Let $\calM_\calQ$ denote the set of all isomorphism classes of objects of $\Rep_k\calQ$. As vector space, the Hall algebra $\bfH_{\calQ}$ is 
\begin{align}
	\bfH_{\calQ}\coloneqq \{ h\colon \calM_\calQ \to \C\; |\; |\mathsf{supp}(h)| <\infty\}\ .
\end{align}
We will denote by $[V]$ the characteristic function of an object $V$. The function $\underline{\dim}$ induces a natural gradation
\begin{align}\label{E:defHall1}
\bfH_{\calQ}=\bigoplus_{d \in \N^I} \bfH_\calQ[d]
\end{align}
with finite dimensional graded pieces. Note that when $I$ is infinite, $\bfH_\calQ[d]=0$ whenever $d$ has infinitely many nonzero components. We denote by $\Z_0^I \subset \Z^I$ the subset of vectors having finitely many nonzero components.

The multiplication in $\bfH_{\calQ}$ is defined as follows. For $M \in \bfH_{\calQ}[d]$ and $M' \in \bfH_{\calQ}[d']$ we set
\begin{align}
	[M] \star [M'] \coloneqq \upsilon^{\langle d, d'\rangle}\sum_{N} P^{N}_{M,M'} [N]
\end{align}
where $N$ ranges over the (finite) set of extensions of $M'$ by $M$ and where
\begin{align}
	P^N_{M,M'}\coloneqq |\{ L \subset N\;|\; L \simeq M', N/L \simeq M\}|\ .
\end{align}
Note that the multiplication is graded with respect to the decomposition \eqref{E:defHall1}. We let $\bfH^{\mathsf{sph}}_{\calQ}$ be the subalgebra of $\bfH_{\calQ}$ generated by the elements $[S_i]$ for $i \in I$. 

We will only be interested here in the following three types of quivers : the equioriented finite type $A$ quivers $A_n$, which has as vertex set ${1, \ldots, n-1}$ and as edge set $i\mapsto i+1$ for $i =1, \ldots,  n-2$; the infinite quiver $A_\infty$ with vertex set $\Z$ and edge set $i \mapsto i+1$; and the cyclically oriented affine type A quiver $A_{n}^{(1)}$ which has as vertex set $\Z/n\Z$ and edge set $i\mapsto i+1$ for $i \in \Z/n\Z$.  

The following result is well-known and due to Ringel.
\begin{theorem}
\hfill
\begin{itemize}\itemsep0.2cm
	\item[(i)] Assume that $\calQ =A_\star$ with $\star \in \N\cup \{\infty\}$. The assignment $[S_{i}] \mapsto \upsilon^{-1/2}E_i$ for all vertices $i$ defines an isomorphism of graded algebras $\bfH_{\calQ} \simeq \bfU^+_{\upsilon}(\fraksl(\star))$.
	\item[(ii)] Assume that $\calQ=A_n^{(1)}$. The assignment $[S_i] \mapsto \upsilon^{-1/2}E_i$ for all vertices $i$ defines an isomorphism of graded algebras $\bfH^{\mathsf{sph}}_{\calQ} \simeq \bfU^+_{\upsilon}(\widehat{\fraksl}(n))$. 
\end{itemize}	
\end{theorem}

We will need a slight variant of the above result, involving $\bfU^-_{\upsilon}$ instead of $\bfU^+_{\upsilon}$. 
\begin{corollary}\label{c:Ringel2}
\hfill
\begin{itemize}\itemsep0.2cm
	\item[(i)] Assume that $\calQ =A_\star$ with $\star \in \N\cup \{\infty\}$. The assignment $[S_{i}] \mapsto -\upsilon^{-1/2}F_i$ for all vertices $i$ defines an isomorphism of graded algebras $\bfH_{\calQ} \simeq \bfU^-_{\upsilon}(\fraksl(\star))$.
	\item[(ii)] Assume that $\calQ=A_n^{(1)}$. The assignment $[S_i] \mapsto -\upsilon^{-1/2}F_i$ for all vertices $i$ defines an isomorphism of graded algebras $\bfH^{\mathsf{sph}}_{\calQ} \simeq \bfU^-_{\upsilon}(\widehat{\fraksl}(n))$. 
\end{itemize}	
\end{corollary}
\begin{proof} 
	This comes from the algebra isomorphism $\bfU^+_\upsilon \simeq \bfU^-_{\upsilon}$ given by $E_J \mapsto -F_J$.
\end{proof}

Because of the relation with quantum groups, it is customary to consider the \emph{twisted} Hall algebra, which is defined as a semi-direct product
\begin{align}
	\bfH^{\mathsf{tw}}_{\calQ} \coloneqq  \bfH^0_{\calQ} \ltimes \bfH_{\calQ}\ ,
\end{align}
where 
\begin{align}
	\bfH^0_\calQ\coloneqq \C[\sfK_0(\Rep_k\calQ)]
\end{align}
is the group algebra of $\sfK_0(\Rep_k\calQ)$. Thus $\bfH^0_\calQ$ has a basis $\{K_d\}$ where $d$ runs among the sets $\Z^n$, $\Z_0^\Z$ or $\Z^{\Z/n\Z}$ for $\calQ=A_n$, $\calQ=A_\infty$ or $\Q=A_n^{(1)}$ respectively. The action of $\mathbf{H}^0_\calQ$ on $\mathbf{H}_\calQ$ is given by
\begin{align}
	K_d\, x\, K_d^{-1}=\upsilon^{(d,l)}x \quad \textit{for all } x \in \bfH_\calQ[l]\ .
\end{align}

We are now ready to recall Varagnolo and Vasserot's geometric interpretation of Hayashi's folding procedure\footnote{Note, however, that \cite{art:varagnolovasserot1999} use the opposite multiplication for Hall algebras and work over a field $\F_{q^2}$.} setting of \cite[Proposition~6.1]{art:varagnolovasserot1999}. Fix some $n >1$. Let us denote by $i \mapsto \overline{i}$, resp.\ $d \mapsto \overline{d}$ the projections $\Z \to \Z/n\Z$, resp.  $\N^{\Z} \to \N^{\Z /n\Z}$. Following \cite[Section~6]{art:varagnolovasserot1999} we will now construct a family of maps 
\begin{align}
	\gamma_d\colon \bfH_{A_n^{(1)}}[\overline{d}] \to \bfH_{A_\infty}[d]
\end{align}
for every $d\in \N^{\Z}$. Fix some $d \in \N^{\Z}$ and a $\Z$-graded vector space $V=\bigoplus_\imath V_\imath$ of dimension $d$. Put:
\begin{align}
	 \overline{V}=\bigoplus_{a \in \Z/n\Z} \overline{V}_a \ ,\quad \overline{V}_a=\bigoplus_{\overline{\jmath}=a} V_\jmath\ .
\end{align}
Thus, even if they may be canonically identified, $V$ is a $\Z$-graded vector space while $\overline{V}$ is $\Z/n\Z$-graded. For any $\imath \in \Z$ set
\begin{align}
	V_{\geq \imath}\coloneqq \bigoplus_{\jmath \geq \imath} V_\jmath\ ,
\end{align}
and denote by $V_{\geq}^{\bullet}$ the associated filtration of $V$. There is an induced filtration $\overline{V}_{\geq}^{\bullet}$ of $\overline{V}$ where
\begin{align}
	\overline{V}_{\geq \imath}=\bigoplus_a \overline{V}_{\geq \imath, \, a}\quad \text{and}\quad \overline{V}_{\geq \imath,\,  a}=\bigoplus_{\jmath \geq \imath,\, \overline{\jmath}=a} V_\jmath\ .
\end{align}

The filtrations $V_{>}^{\bullet},\overline{V}_{>}^{\bullet}$ are defined in the same way, replacing $\geq$ by $>$. The associated graded space
\begin{align}
	\mathsf{gr}(\overline{V})\coloneqq \bigoplus_\imath (\overline{V}_{\geq \imath}/ \overline{V}_{>\imath})
\end{align}
is canonically isomorphic to $V$ as a $\Z$-graded vector space. 

Let $E_V, E_{\overline{V}}$ denote the vector spaces of representations of $A_\infty$ and $A_n^{(1)}$ in $V$, resp.\ $\overline{V}$. Let also $E_{V,\overline{V}} \subset E_{\overline{V}}$ be the subset of representations preserving the filtration $\overline{V}^\bullet_{\geq}$. Let us write $j \colon E_{V,\overline{V}} \to E_{\overline{V}}$ for the embedding. Taking the associated graded yields a map $p\colon E_{V,\overline{V}} \to E_V$.

We set
\begin{align}
	\gamma_d\coloneqq \upsilon^{-h(d)}\, p_! \circ j^\ast \colon  \bfH_{A_n^{(1)}}[\overline{d}]\to \bfH_{A_\infty}[d]\ ,
\end{align}
where
\begin{align}
	h(d)\coloneqq -\sum_{\ell<0} \langle d, \tau^\ell (d)\rangle = \sum_{\genfrac{}{}{0pt}{}{\imath<\jmath}{\overline{\imath}=\overline{\jmath}}}\,  d_\imath(d_{\jmath+1}-d_\jmath)\ ,
\end{align}
with $\tau\colon (d_i)_i \mapsto (d_{i-1})_i$ being the translation operator.
Finally, put
\begin{align}
	k(h,h')\coloneqq \sum_{\ell >0} (h,\tau^{\ell}(h'))=\sum_{\genfrac{}{}{0pt}{}{\imath>\jmath}{\overline{\imath}=\overline{\jmath}}}\, h_\imath(2h_\jmath -h_{\jmath-1} - h_{\jmath+1}).
\end{align}

\begin{proposition}[{\cite[Proposition~6.1]{art:varagnolovasserot1999}}]\label{P:Fockfinite}
	For any pair $l,l' \in \N^{\Z /n\Z}$ and any $d\in \N^\Z$ such that $\overline{d}=l + l'$ we have
	\begin{align}
		\sum_{\genfrac{}{}{0pt}{}{h+h'=d}{\overline{h}=l,\, \overline{h'}=l'}} \, \upsilon^{k(h,h')}\gamma_h(u) \star \gamma_{h'}(u')=\gamma_d(u \star u')
	\end{align}
	for any $u \in \bfH_{A_n^{(1)}}[l], u' \in \bfH_{A_n^{(1)}}[l']$.
\end{proposition}

This proposition will be used in the following way. For $l \in \N^{\Z/n\Z}$ and  $x \in \bfH_{A_n^{(1)}}[l]$ we set
\begin{align}
	r(x)\coloneqq \sum_{d:\, \overline{d}=l} K_{o(d)} \, \gamma_d(u)  
\end{align}
where $o(d)\coloneqq \sum_{\ell <0} \tau^{\ell}(d)$. Note $r(x)$ takes values in a completion $\bfH^{\mathsf{tw,c}}_{A_\infty}$ of $\bfH^{\mathsf{tw}}_{A_\infty}$, which we now define.  Let $\bfH^{0,\mathsf{c}}_{A_\infty}$ be the group algebra of $\Z^\Z$. In other words $\bfH^{0,\mathsf{c}}_{A_\infty}$ is generated by elements $K_d$ for $d \in \Z^\Z$ (possibly having infinitely many nonzero components) subject to the relations $K_d K_l=K_{d+l}$, for $d,l \in \Z^\Z$. We set
\begin{align}
	\bfH^{\mathsf{tw,c}}_{A_\infty}\coloneqq \bigoplus_{l \in \N^{\Z/n\Z}}\bfH^{\mathsf{tw,c}}_{A_\infty} [l] \quad \textit{and}\quad
	\bfH^{\mathsf{tw,c}}_{A_\infty}[l] \coloneqq \prod_{\substack{d \in \N_0^\Z \\ \overline{d}=l}}\left\{\bfH^{0,\mathsf{c}}_{A_\infty}\ltimes \bfH_{A_\infty}[d]\right\}\ .
\end{align}
Observe that for any $m \in \Z^\Z$ and any $d \in \N^\Z_0$ the Euler form $(m, d)$ is well-defined, so that the semi-direct product above makes sense. Moreover, $\bfH^{\mathsf{tw,c}}_{A_\infty}$ is still an algebra since there are only finitely many ways to break up a dimension vector $d$ as a sum $d' + d''$. 

\begin{corollary}\label{c:rhomfinite}
	The map $r\colon \bfH_{A_n^{(1)}} \to \bfH^{\mathsf{tw,c}}_{A_\infty}$ is an algebra homomorphism.
\end{corollary}

Hayashi's action of $\bfU_\upsilon^-(\hatfraksl(n))$ on the Fock space $\calF_\Z$ can now be recovered as follows. By Corollary~\ref{c:Ringel2} we embed $\bfU_\upsilon^-(\hatfraksl(n))$ into $\bfH_{A_n^{(1)}}$ and use the homomorphism $r$ to pullback the representation of $\bfU_\upsilon^-(\fraksl(\infty)) \simeq \bfH_{A_\infty}$ on $\calF_\Z$. Note that since the Fock space representation of $\bfU_\upsilon^-(\hatfraksl(n))$ is faithful, the map $r$ is injective; this can also be seen directly.

\bigskip\section{Fock space representation of the quantum group of the line}\label{s:Fockspaceline}

Fix $\K\in \{\Z, \Q, \R\}$. We denote by $\fun{\K}$ the algebra of piecewise constant, right--continuous functions $f\colon \R\to \Z$, with finitely many points of discontinuity, bounded support and whose points of discontinuity belong to $\K$. This means that $f\in\fun{\K}$ if and only if $f=\sum_J\,c_{J}\,\cf_J$, where the sum runs over all intervals of $\K$ and $c_J\in\Z$ is zero for all but finitely many $J$. Given $f,g \in \fun{\K}$, we define
\begin{align}\label{eq:eulerform}
	\langle f, g\rangle \coloneqq\sum_x\, f_{-}(x)(g_{-}(x)-g_{+}(x))\quad \text{and}\quad (f,g)\coloneqq\langle f, g\rangle+\langle g, f\rangle\ ,
\end{align}
where we have set $h_{\pm}(x)\coloneqq\lim_{t \to 0, t >0} h(x\pm t)$. Let $\fun{\K}^+$ (resp.\ $\fun{\K}^-$) be the set of functions $f\in \fun{\K}$ such that $f(x)\geq 0$ (resp.\ $f(x)\leq 0$) for any $x\in \R$.

\begin{remark}\label{rem:eulerform}
Let $J_i=[a_i, b_i[$ be a $\K$-interval for $i=1, 2$. Then
\begin{itemize}\itemsep0.2cm
	\item $\langle \cf_{J_1}, \cf_{J_2}\rangle =1$ if one of the following conditions holds: $J_1=J_2$, $a_1=a_2$  and $b_2<b_1$, $a_2<a_1$ and $b_1=b_2$, $a_2<a_1<b_2<b_1$;
	\item  $\langle \cf_{J_1}, \cf_{J_2}\rangle =0$ if one of the following conditions holds: $\overline J_1\cap \overline J_2=\emptyset$, $b_2=a_1$, $a_1=a_2$ and $b_1<b_2$, $a_1<a_2$ and $b_1=b_2$, $a_1<a_2<b_2<b_1$, $a_2<a_1<b_1<b_2$;
	\item $\langle \cf_{J_1}, \cf_{J_2}\rangle =-1$ if either $b_1=a_2$ or $a_1<a_2<b_1<b_2$.
\end{itemize}	
Thus, the following holds:
	\begin{itemize}\itemsep0.2cm
		\item for any $\K$-interval $J$,  we have $(\cf_J, \cf_J)=2$;
		\item if $J_1, J_2$ are $\K$-intervals of the form $J_1=[a,b[$ and $J_2=[b,c[$, we get
		\begin{align}
			(\cf_{J_1}, \cf_{J_2})=(\cf_{J_2}, \cf_{J_1})=-1 \ ;
		\end{align}
		\item in any other case, (i.e. if $I \neq J$ are $\K$-intervals such that $\cf_{J} + \cf_{I}$ is not the characteristic function of any interval) then $(\cf_{I}, \cf_{J})=0$.
	\end{itemize}
\end{remark}

\subsection{The bialgebra $\bfU_\upsilon(\fraksl(\K))$}\label{ss:algebra-K}

\begin{definition}[{cf.\ \cite[Section~1.1]{art:salaschiffmann2017}}]\label{def:R}
	Let $\bfU_\upsilon(\fraksl(\K))$ be the topological $\widetilde{\Q}$-bialgebra generated by elements $E_J, F_J,$ $K_{J}^{\pm 1}$, where $J$ runs over all $\K$-intervals, modulo the following set of relations:
	\begin{itemize}\itemsep0.4cm
		\item \emph{Drinfeld-Jimbo relations}:
		
		\begin{itemize}\itemsep0.2cm
			\item for any intervals $J, I, I_1, I_2$,
			\begin{align}
				[K_{I_1},K_{I_2}] &=0\ ,\\
				K_{I}\, E_{J}\, K_{I}^{-1} &=\upsilon^{( \cf_{I},\cf_{J})}\, E_{J}\ ,\\[2pt]
				K_{I}\, F_{J}\, K_{I}^{-1} &=\upsilon^{-( \cf_{I},\cf_{J})}\, F_{J} \ ;
			\end{align}
		
			\item if $J_1, J_2$ are intervals such that $J_1 \cap J_2 = \emptyset$, 
			\begin{align}
				[F_{J_1},E_{J_2}]=0\ ;
			\end{align}
		
			\item for any interval $J$,
			\begin{align}
				[E_J,F_J]=\frac{K_J-K_J^{-1}}{\upsilon-\upsilon^{-1}} \ ;
			\end{align} 
		\end{itemize}
		
		\item \emph{join relations}: if $J_1,J_2$ are intervals of the form $J_1=[a,b[$ and $J_2=[b,c[$ then
		\begin{align}
			K_{J_1}\, K_{J_2}=K_{J_1\cup J_2}\ ;
		\end{align}
		and
		\begin{align}
			\begin{aligned}
			E_{J_1\cup J_2} &=\upsilon^{1/2}\,E_{J_1}\,E_{J_2} - \upsilon^{-1/2}\,E_{J_2}\, E_{J_1}\ ,  \\[2pt]
			F_{J_1\cup J_2} &=\upsilon^{-1/2}\,F_{J_2}\,F_{J_1} - \upsilon^{1/2}\, F_{J_1}\, F_{J_2}  \ ;
			\end{aligned}
		\end{align}
		
		\item \emph{nest relations}:
		\begin{itemize}\itemsep0.2cm
			\item if $J_1, J_2$ are intervals such that $\overline{J_1} \cap \overline{J_2} =\emptyset$,
			\begin{align}\label{eq:nest-1}
				[E_{J_1},E_{J_2}]=0\quad\text{and}\quad [F_{J_1}, F_{J_2}]=0\ ;
			\end{align}
			
			\item if $J_1, J_2$ are intervals such that $J_1\subseteq J_2$,
			\begin{align}
				\begin{aligned}
				\upsilon^{\langle \cf_{J_1}, \cf_{J_2}\rangle} \, E_{J_1}\, E_{J_2} &=\upsilon^{\langle \cf_{J_2}, \cf_{J_1}\rangle} \, E_{J_2}\, E_{J_1} \ ,\\[2pt]
				\upsilon^{\langle \cf_{J_1}, \cf_{J_2}\rangle} \, F_{J_1}\, F_{J_2} &=\upsilon^{\langle \cf_{J_2}, \cf_{J_1}\rangle} \, F_{J_2}\, F_{J_1} \ .
				\end{aligned}
			\end{align}
		\end{itemize}
	\end{itemize}
	
	The coproduct is given by:
	\begin{align}
		\begin{aligned}
		 \Delta(K_J)&=K_J\otimes K_J\ ,\\[2pt]
		 \Delta (E_{[a,\, b[})&=E_{[a,\, b[} \otimes 1 + \sum_{a<c<b} \upsilon^{-1/2}\, (\upsilon-\upsilon^{-1})\, E_{[a,\, c[}\, K_{[c,\, b[} \otimes E_{[c,\, b[} + K_{[a,\, b[} \otimes E_{[a,\, b[} \ ,\\[2pt]
		 \Delta (F_{[a,\, b[})&=1\otimes F_{[a, \, b[} - \sum_{a<c<b} \upsilon^{-1/2}\, (\upsilon-\upsilon^{-1})\, F_{[c,\, b[}\otimes F_{[a,\, c[}\, K_{[c,\, b[}^{-1}  + F_{[a,\, b[}\otimes K_{[a,\, b[}^{-1}\ .
		\end{aligned}
	\end{align}
	Here the sums on the right-hand-side run over all possible\footnote{$\bfU_\upsilon(\fraksl(\K))$ is a \textit{topological} coalgebra, i.e. the comultiplication only takes values in a suitable completion, see \cite{art:salaschiffmann2017}; we will not use the coproduct in this paper.} $\K$-values $c\in [a, b[$.
\end{definition}

\begin{definition}
	Let $f=\sum_J\, c_J\, \cf_J\in \fun{\K}$. Then we set $K_f\coloneqq \prod_J\, K_J^{\, c_J}$.	
\end{definition}

\begin{remark}\label{rem:Serre}
	Let $J_1, J_2$ be $\K$-intervals of the form $J_1=[a,b[$ and $J_2=[b,c[$ (so that $J_1 \cup J_2$ is again a interval). Then, by the join and the nest relations we derive:
	\begin{align}\label{eq:serre-1}
		E_{J_1}E_{J_2}^2-[2]\, E_{J_2}E_{J_1}E_{J_2}+E_{J_2}^2 E_{J_1}&=0\\
		F_{J_1}F_{J_2}^2-[2]\, F_{J_2}F_{J_1}F_{J_2}+F_{J_2}^2 F_{J_1} &=0	
	\end{align}
	and similarly with $J_1$ and $J_2$ exchanged. In other words, the usual (type A) cubic Serre relations are implied by the join and nest relations.
\end{remark}

\begin{definition}
	We define a $\fun{\K}$-gradation on $\bfU_\upsilon(\fraksl(\K))$ by setting
	\begin{align}
		\deg(E_J)=\cf_J\ ,\quad \deg(F_J)=-\cf_J \quad\text{and}\quad \deg(K_J)=0
	\end{align}
	for any $\K$-interval $J$.
\end{definition}

As usual, we define $\bfU_\upsilon^+(\fraksl(\K))$ (resp.\ $\bfU_\upsilon^{-}(\fraksl(\K))$) as the subalgebra of $\bfU_\upsilon(\fraksl(\K))$ generated by the $E_J$ (resp.\ the $F_J$) for any $\K$-interval $J$. Let also $\bfU_\upsilon^0(\fraksl(\K))$ be the subalgebra of $\bfU_\upsilon(\fraksl(\K))$ generated by the $K_J^{\pm 1}$ for any $\K$-interval $J$. Finally, set
\begin{align}
	\bfU_\upsilon^{\leq 0}(\fraksl(\K))\coloneqq \bfU_\upsilon^{-}(\fraksl(\K))\cdot \bfU_\upsilon^0(\fraksl(\K)) \quad\text{and}\quad \bfU_\upsilon^{\geq 0}(\fraksl(\K))\coloneqq \bfU_\upsilon^0(\fraksl(\K))\cdot \bfU_\upsilon^+(\fraksl(\K)) \ .
\end{align}
Then $\bfU_\upsilon^{\leq 0}(\fraksl(\K))$ and $\bfU_\upsilon^{\geq 0}(\fraksl(\K))$ may be endowed with the structure of topological $\widetilde{\Q}$-bialgebras.

\subsection{The Fock space $\calF_\K$ of $\bfU_\upsilon(\fraksl(\K))$}

Our aim in this section is to define explicitly an action of the quantum group $\bfU_\upsilon(\fraksl(\K))$ on the space
\begin{align}
	\calF_\K\coloneqq \bigoplus_{p \in \mathsf{Pyr}(\K)}\,  \widetilde \Q \vert p\rangle,
\end{align}
generalizing the standard Fock space representation of $\bfU_{\upsilon}(\fraksl(\infty))$.

\begin{theorem} \label{thm:Fock-R}
	The following formulas define an action of the quantum group $\bfU_\upsilon(\fraksl(\K))$ on $\calF_\K$: for any $J=[a,b[$ and $p \in \mathsf{Pyr}(\K)$
	\begin{align}
		E_J \vert p\rangle&\coloneqq \begin{cases}
		 -\upsilon^{1/2} (-\upsilon)^{p(b)-p(a)}\, \vert p-\cf_J\rangle &\text{if $J$ is a removable interval of $p$}\ ,\\[4pt]
		0 & \text{otherwise}\ ,
		\end{cases}\\[6pt]
		F_J \vert p\rangle&\coloneqq \begin{cases}
		\upsilon^{1/2} (-\upsilon)^{p(a)-p(b)} \, \vert p+\cf_J\rangle &\text{if $J$ is an addable interval of $p$}\ ,\\[4pt]
		0 & \text{otherwise}\ ,
		\end{cases}\\[6pt]
		K_J \vert p\rangle&\coloneqq \upsilon^{n_{J}(p)}\, \vert p\rangle\ ,
	\end{align}
	where
	 \begin{align}
		n_{J}(p) =
		\begin{cases}
		0 & \text{if $J$ is neither addable or removable to $p$}\ , \\
		1 & \text{if $J$ is addable to $p$}\ ,\\
		-1 & \text{if $J$ is removable to $p$}\ .
		\end{cases} 
	 \end{align}
	 The representation $\calF_{\K}$ is highest weight and irreducible.
 \end{theorem}

 \begin{proof} For simplicity, let us put
	\begin{align}
		e_{J,p}\coloneqq -\upsilon^{1/2} (-\upsilon)^{p(b)-p(a)}\quad \text{and} \quad f_{J,p}\coloneqq \upsilon^{1/2} (-\upsilon)^{p(a)-p(b)}\ .
	\end{align}
	We will check the compatibility with all the defining relations directly.

	
	\emph{Nest relations}. Let's start by verifying the nest relations for $F_J$. We begin with the case of a pair of intervals $I,J$ such that $\overline{I} \cap \overline{J}=\emptyset$. Then it is possible to successively add $I$ and then $J$ to $p$ if and only if it is possible to successively add $J$ and then $I$ to $p$. Moreover, in this case we have $f_{J, p+\cf_I}=f_{J,p}$ and $f_{I,p+\cf_J}=f_{I,p}$. It follows that $F_I\, F_J \vert p\rangle=F_J\, F_I\vert p\rangle$ as wanted.
	
	Now assume that $I=[a,b[ \subseteq J=[a',b'[$. We claim that it is possible to add both $I$ and $J$ to $p$ (in either order) only if $\overline{I} \subset J^{\circ}$, i.e., only if $a'<a<b<b'$. Indeed, suppose for instance that $a'=a$. Then $(p+\cf_I+\cf_J)(a)=p(a) + 2$ while $(p+\cf_I+\cf_J)_{-}(a)=p_{-}(a)$. But then $p$ and $p+\cf_I + \cf_J$ cannot both be pyramids: if $a \leq 0$ then we have $p_{-}(a) \leq p(a)$ hence $(p+\cf_I + \cf_J)_{-}(a) \leq (p+\cf_I+\cf_J)(a) -2$, violating condition iii) of pyramids; likewise, if $a >0$ then $p(a) \leq p(a) +1$ hence $(p+\cf_I + \cf_J)_{-}(a) \leq (p+\cf_I+\cf_J)(a)-1$, again violating condition iii) of pyramids. A very similar reasoning takes care of the case $b=b'$. We are thus left with the case $a'<a<b<b'$. In this situation, it is easy to see that either $I$ and $J$ are both addable to $p$, in which case they can be added in either order, or one of them is not addable to $p$, in which case the two can not be added (in either order). If both may be added, then we have
	\begin{align}
		p(a)-p(b)=(p+\cf_J)(a)-(p+\cf_J)(b)\quad\text{and}\quad p(a')-p(b')=(p+\cf_I)(a')-(p+\cf_I)(b')\ ,
	\end{align}
	from which it immediately follows that $f_{I,p+\cf_J}\, f_{J,p}=f_{J,p+\cf_I}\, f_{I,p}$, i.e., $F_I\, F_J\vert p\rangle=F_J\, F_I\vert p\rangle$  (note that $\langle \cf_I,\cf_J\rangle=\langle \cf_J,\cf_I \rangle=0$, cf.\ Remark~\ref{rem:eulerform}). The proof of the nest relations for the $E_J$ follows similarly as above. 
	
	\vspace{.1in}
	
	\emph{Join relations}. Let's now move to verify the join relations for the $F_J$. Assume that $I=[a,b[, J=[b,c[$. There are three mutually exclusive possible situations: it is not possible to add both $I$ and $J$ to $p$ (in either order) and neither $I\cup J$; it is possible to add $I$, then $J$, hence also $I\cup J$; it is possible to add $J$ then $I$, hence also $I\cup J$. In the first case, we have $F_I\, F_J\vert p\rangle=0=F_J\, F_I\vert p\rangle$, hence the join relation is proved. In the second case, we have
	\begin{align}
		(p+\cf_I)(b)=p(b)\ , \quad (p+\cf_I)(c)=p(c)\ ,
	\end{align}
	hence
	\begin{align}
		\upsilon^{-1/2}f_{J,p+\cf_I}\,f_{I,p}=\upsilon^{1/2} (-\upsilon)^{p(a)-p(c)} = f_{I\cup J, p}\ .
	\end{align}
	In the last case, we have
	\begin{align}
		(p+\cf_J)(a)=p(a)\ , \quad (p+\cf_J)(b)=p(b)+1\ .
	\end{align}
	Hence
	\begin{align}
		-\upsilon^{1/2}f_{I,p+\cf_J}\, f_{J,p}=-\upsilon^{3/2} (-\upsilon)^{p(a)-p(c)-1}=   f_{I\cup J, p}  \ .
	\end{align}
	In both cases, the join relation $F_{I \cup J}\vert p\rangle=\upsilon^{1/2}F_I\, F_J\vert p\rangle-\upsilon^{-1/2} F_J\, F_I \vert p\rangle$ is verified. Similarly, we verify the join relation for $E_J$.
	
	
	\emph{Drinfeld-Jimbo type relations}. Finally, let us verify the Drinfeld-Jimbo type relations. The only relation we have to address carefully is the commutation relation between the $E_J$ and the $F_J$. By the same argument as at the beginning of this proof, it is not possible to both add and remove the same interval $J$ to a pyramid $p$ (otherwise it would be possible to add twice the interval $J$ to $p-\cf_J$). Thus either $F_J\, E_J\vert p\rangle=0$ or $E_J\, F_J\vert p\rangle=0$. There are three (mutually exclusive) possibilities:
	\begin{itemize}[leftmargin=0.5cm]\itemsep0.2cm
		\item it is neither possible to add nor remove $J$. The assertion follows. 
		
		\item it is possible to add $J=[a,b[$. Then 
		\begin{align}
			E_J\, F_J\vert p\rangle= -\upsilon^{1/2} (-\upsilon)^{(p+\cf_J)(b)-(p+\cf_J)(a)}\upsilon^{1/2}(-\upsilon)^{p(a)-p(b)} \vert p\rangle=-\upsilon(-\upsilon)^{-1} \vert p\rangle=\vert p\rangle\ .
		\end{align}
		Hence, $[E_J, F_J]\vert p\rangle= \vert p\rangle$, as wanted. 
		
		\item it is possible to remove $J$. Then
		\begin{align}
			F_J\, E_J\vert p\rangle =\upsilon^{1/2} (-\upsilon)^{(p-\cf_J)(a)-(p-\cf_J)(b)} (-\upsilon^{1/2}) (-\upsilon)^{p(b)-p(a)}\vert p\rangle=-\upsilon (-\upsilon)^{-1} \vert p\rangle\ .
		\end{align}
		Hence, $[E_J, F_J]\vert p\rangle= - \vert p\rangle$, as expected.
	\end{itemize}
	The rest of the Drinfeld-Jimbo relations are easier to be verified and we leave the check to the interested reader. 
	
	To finish, let us check that $\calF_\K$ is generated by $\vert 0 \rangle$ and is irreducible. Let $p$ be a $\K$-pyramid. We may write $p$ as a sum $p=\sum_{i=1}^s \cf_{I_i}$ where $I_1, \ldots I_s$ are strictly nested intervals, i.e. $I_1 \supset I_2 \supset \cdots \supset I_s$ and the endpoints of the $I_i$ are all distinct. It is easy to see that, up to a constant, $F_{I_s} \cdots F_{I_1} \cdot \vert 0 \rangle=\vert p \rangle$, proving the first assertion. The irreducibility may be proved by reversing this argument: up to a constant, we have $E_{I_1} \cdots E_{I_s} \cdot \vert q \rangle = \delta_{p,\, q} \vert 0 \rangle$ if $q$ is any pyramid such that $|p|=|q|$.
\end{proof}

\begin{remark}
	Note that for $J=[a,b[$ we have
	\begin{align}
		\langle \cf_J, p\rangle = \sum_{x>a}^b (p_{-}(x)-p_+(x))=p(a)-p(b),\quad \langle p, \cf_J\rangle = p_{-}(b)-p_{-}(a)\ .
	\end{align}
	Hence 
	\begin{align}
		e_{J, p}= -\upsilon^{1/2} (-\upsilon)^{-\langle \cf_J, p\rangle}\ , \quad f_{J, p}= \upsilon^{1/2} (-\upsilon)^{\langle \cf_J, p\rangle}\ .
	\end{align}
\end{remark}

For completeness, we also state the following result.
\begin{lemma}\label{lem:K}
	Let $J$ be a $\K$-interval and $p$ be a $\K$-pyramid. Then
	\begin{align}
		n_{J}(p) =\delta_{0 \in J} -(\cf_J,p)\ .
	\end{align}
\end{lemma}

\begin{proof} Let $J=[a,b[$. Note that $(\cf_J,p)=(p_{-}(b)-p(b))- (p_{-}(a)-p(a))$. There are three (mutually exclusive) possibilities to consider:
	\begin{itemize}[leftmargin=0.5cm]\itemsep0.2cm
		\item it is neither possible to add nor remove $J$. We will check that $\delta_{0 \in J} -(\cf_J,p)=0$. There are two subcases to consider, according to whether $0 \in J$ or $0 \not\in J$.
		\begin{itemize}\itemsep0.2cm
			\item[(a')] $0 \not\in J$. In that case, either the two endpoints $a$ and $b$ of $J$ do not lie on points of discontinuity of $p$ or they both lie on points of discontinuity of $p$. Thus $(\cf_J,p)=0$.
			\item[(a'')] $0 \in J$. In that case, exactly one of the endpoints $a$ and $b$ of $J$ lie on a point of discontinuity of $p$. Moreover, $a \leq 0$ and $b >0$ so that $p_{-}(a) \leq p(a)$ while $p_{-}(b) \geq p(b)$. It follows that $(\cf_J,p)=1$.
		\end{itemize}
		\item it is possible to add $J$. Then we will show that $\delta_{0 \in J} -(\cf_J,p)=1$. There are again two subcases to consider, according to wether $0 \in J$ or $0 \not\in J$. 
		\begin{itemize}\itemsep0.2cm
			\item[(b')] $0 \not\in J$. In that case, either $a>0$ and $a$ is a point of discontinuity of $p$ but not $b$, or $b<0$ and $b$ is a point of discontinuity of $p$ but not $a$. Hence $(\cf_J,p)=-1$.
			\item[(b'')] $0 \in J$. In that case, neither endpoints $a$ and $b$ of $J$ may lie over a point of discontinuity of $p$, and thus $(\cf_J,p)=0$.
		\end{itemize}
		\item it is possible to remove $J$. Then we will show that $\delta_{0 \in J} -(\cf_J,p)=-1$. There are two subcases to consider, according to whether $0 \in I$ or $0 \not\in J$. 
		\begin{itemize}\itemsep0.2cm
			\item[(c')] $0 \not\in J$. Here, either $a<0$ and $a$ lies over a point of discontinuity of $p$ (but not $b$), or $a>0$ and $b$ lies over a point of discontinuity of $p$ (but not $a$). We deduce $(\cf_J,p)=1$.
			\item[(c'')] $0 \in J$. In that case, $a\leq 0, b>0$ and both $a$ and $b$ lie over points of discontinuity of $p$, so that $(\cf_J,p)=2$.
		\end{itemize}
	\end{itemize}
\end{proof}

\subsection{Generalisation to arbitrary discrete subsets}\label{ss:gen}

Let $\alpha\colon \Z \to \R$ be any strictly increasing map. We may define a quantum group $\bfU_{\upsilon}(\fraksl(\alpha(\Z)))$ as the subalgebra (but not coalgebra) of $\bfU_\upsilon(\fraksl(\R))$ generated by all elements $E_J,F_J,K_J$ where the endpoints of $J$ belong to $\alpha(\Z)$. There is an obvious isomorphism $\gamma_{\alpha}\colon \bfU_\upsilon(\fraksl(\Z)) \stackrel{\sim}{\to} \bfU_\upsilon(\fraksl(\alpha(\Z)))$ mapping $E_J, F_J, K_J$ to $E_{\alpha(J)}, F_{\alpha(J)}, K_{\alpha(J)}$ respectively. Let us now assume that $0 \in \mathsf{Im}(\alpha)$ and let $\calF_{\alpha(\Z)}$ be the linear subspace of $\calF_\R$ spanned by pyramids $p$ satisfying $D_\R(p) \in \mathsf{Im}(\alpha)$. It is clear that $\calF_{\alpha(\Z)}$ is stable under the action of $\bfU_\upsilon(\fraksl(\alpha(\Z)))$ and that the following square is commutative:
\begin{align}
	\begin{tikzcd}[ampersand replacement = \&]
		\bfU_\upsilon(\fraksl(\alpha(\Z))) \ar{r} \& \mathsf{End}(\calF_{\alpha(\Z)}) \\ 
		\bfU_\upsilon(\fraksl(\Z)) \ar{u}{\gamma_\alpha} \ar{r} \& \mathsf{End}(\calF_{\Z}) \ar{u}{\iota_{\alpha}}
	\end{tikzcd}\ ,
\end{align}
where $\iota_{\alpha}\colon \calF_{\Z} \stackrel{\sim}{\to} \calF_{\alpha(\Z)}$ is the linear isomorphism induced by $\alpha$.

\subsection{Relation with $\bfU_\upsilon(\fraksl(\infty))$}\label{ss:infty}

The quantum group $\bfU_\upsilon(\fraksl(\Z))$ admits a minimal set of generators $\{E_{[i,\, i+1[},$ $F_{[i,\, i+1[},$ $K_{[i,\, i+1[}^{\pm 1}$ $\vert\;i\in\Z \}$. Thanks to the nest relation \eqref{eq:nest-1} and the Serre relations \eqref{eq:serre-1}, the assignment
\begin{align}
	E_i\mapsto E_{[i,\, i+1[} \ , \quad F_i \mapsto F_{[i,\, i+1[}\ , \quad  K_i^{\pm 1}\mapsto K_{[i,\, i+1[}^{\pm 1}
\end{align}
defines an isomorphism of bialgebras $\bfU_\upsilon(\fraksl(\infty)) \to \bfU_\upsilon(\fraksl(\Z))$.

The action of $\bfU_\upsilon(\fraksl(\Z))$ on $\calF_\Z$ given in Theorem~\ref{thm:Fock-R} reduces to
\begin{align}
	E_{[i,\, i+1[} \vert \lambda \rangle&\coloneqq \begin{cases}
	-\upsilon^{1/2}\, \vert \nu \rangle & \text{if $Y_\lambda\setminus Y_\nu$ is a box with color $i$ and $i<0$}\ , \\[4pt]
	\upsilon^{-1/2}\, \vert \nu \rangle & \text{if $Y_\lambda\setminus Y_\nu$ is a box with color $i$ and $i\geq 0$}\ , \\[4pt]
	0 & \text{otherwise} \ ,
	\end{cases}\\[6pt]
	F_{[i,\, i+1[} \vert \lambda \rangle&\coloneqq\begin{cases}
	-\upsilon^{-1/2}\,\vert \mu \rangle & \text{if $Y_\mu\setminus Y_\lambda$ is a box with color $i$ and $i<0$} \ ,\\[4pt]
	\upsilon^{1/2}\,\vert \mu \rangle & \text{if $Y_\mu\setminus Y_\lambda$ is a box with color $i$ and $i\geq 0$} \ ,\\[4pt]
	0 & \text{otherwise}\ ,
	\end{cases}\\[6pt]
	K_{[i,\, i+1[} \vert  \lambda \rangle&\coloneqq\upsilon^{n_i(\lambda)}\vert \lambda \rangle
\end{align}
for $i\in \Z$. Here, we used the bijection between $\Z$-pyramids and partitions (cf.\ Lemma~\ref{lem:bijection}) and the identity
\begin{align}
	n_i(\lambda)=n_{[i, i+1[}(p_\lambda) \ ,
\end{align}
where $p_\lambda$ is the $\Z$-pyramid associated with $\lambda$. Thus the $\bfU_\upsilon(\fraksl(\Z))$-action on $\calF_\Z$ defined in Theorem~\ref{thm:Fock-R} is identified with a suitable rescaling of the standard $\bfU_\upsilon(\fraksl(\infty))$-action on the Fock space.

\bigskip\section{The circle quantum group and the folding procedure}\label{s:Hallalgebra} 
%

In this section, we recall the definition of the circle quantum group and we provide a realization of it in terms of Hall algebras of continuum quivers. Finally, we generalize the ``folding procedure" of Hayashi-Misra-Miwa  (see \cite[Section~6.2]{art:hayashi1990} and \cite[Section~2]{art:misramiwa1990}) to the circle quantum group. We will readapt a interpretation of the folding procedure due to Varagnolo-Vasserot (cf.\  \cite[Section~6]{art:varagnolovasserot1999}) which uses the Hall algebras realization of the quantum group. This folding procedure will be a key tool to the proof of the action of the circle quantum group on the Fock space in the next section. 
 
 Let $\K$ be either $\Q, \R$ or $\alpha(\Z)$ where $\alpha \colon \Z \to \R$ is a strictly increasing map whose image contains $0$ and is invariant under integer translations. Set $S^1_\K\coloneqq \K/\Z$ and denote by $\pi_\K\colon \K \to S^1_\K$ the projection map. 

\subsection{The bialgebra $\bfU_\upsilon(\fraksl(S^1_\K))$}

The notion of (closed on the left, open on the right) intervals generalizes in a straightforward way to $S^1_\K$. We say that an interval $J$ of $S^1_\K$ is \emph{strict} if $J\neq S^1_\K$.

We denote by $\fun{S^1_\K}$ the algebra of piecewise constant, right--continuous, $\Z$-valued functions $f\colon S^1_\R\to \R$, with finitely many points of discontinuity, whose points of discontinuity belong to $S^1_\K$. There is an obvious map $\pi_\K\colon\fun{\K}\to \fun{S^1_\K}$. Formula \eqref{eq:eulerform} defines bilinear forms $\langle \cdot, \cdot \rangle$ and $(\cdot, \cdot)$ on $\fun{S^1_\K}$. One defines $\fun{S^1_\K}^{\pm}$ as before. There is an obvious map $\pi_\K\colon\fun{\K}^{\pm}\to \fun{S^1_\K}^{\pm}$. 

\begin{remark}
	Note that $(\cf_{S^1}, f)=0$ for any $f\in \fun{S^1_\K}$. In fact it is easy to see that the kernel of $(\cdot,\cdot)$ is equal to $\Z\, \cf_{S^1}$.
\end{remark}
Define, for a strict interval $J \subset S^1_{\K}$,
\begin{align}
\Int(J)\coloneqq \big\{I\subset \K\;\vert\; \text{$I=[a,b[$ is a $\K$-interval}\ , \pi_\K(I)=J\, \big\}\ .
\end{align}
Thus $\Int(J)$ consists of all integer translates of some (any) $\K$-interval $\tilde{J}$ such that $\pi_\K(\tilde{J})=J$.
\begin{definition}
	Let $J$ and $J'$ be strict intervals of $S^1_\K$. We say that $J$ is \emph{left adjacent} to $J'$ if $J\cap J'=\emptyset$, $\overline{J}\cap J'\neq \emptyset$ and $J\cup J'$ is an interval of $S^1_\K$. We denote it by $J\to J'$.
\end{definition} 
We are now ready to give the definition of the circle quantum group.
\begin{definition}[{\cite[Definition~1.1]{art:salaschiffmann2017}}]\label{def:S}
	Let $\bfU_\upsilon(\fraksl(S^1_\K))$ be the topological $\widetilde \Q$-bialgebra generated by elements $E_J, F_J, K_{J'}^{\pm 1}$, where $J$ (resp.\ $J'$) runs over all strict intervals (resp.\ intervals) of $S^1_\K$, modulo the following set of relations:
	\begin{itemize}\itemsep0.4cm
		\item \emph{Drinfeld-Jimbo relations}:
		\begin{itemize}\itemsep0.2cm
			\item for any intervals $I, I_1, I_2$ and strict interval $J$,
			\begin{align}\label{eq:DJ-1}
			[K_{I_1},K_{I_2}] &=0\ ,\\ \label{eq:DJ-2}
			K_{I}\, E_{J}\, K_{I}^{-1} &=\upsilon^{( \cf_{I},\cf_{J})}\, E_{J}\ ,\\[2pt] \label{eq:DJ-3}
			K_{I}\, F_{J}\, K_{I}^{-1} &=\upsilon^{-( \cf_{I},\cf_{J})}\, F_{J} \ ;
			\end{align}
			
			\item if $J_1, J_2$ are strict intervals such that $J_1 \cap J_2 = \emptyset$, 
			\begin{align}\label{eq:DJ-4}
			[E_{J_1},F_{J_2}]=0\ ;
			\end{align}
			
			\item for any strict interval $J$,
			\begin{align}\label{eq:DJ-5}
			[E_J,F_J]=\frac{K_J-K_J^{-1}}{\upsilon-\upsilon^{-1}} \ ;
			\end{align} 
		\end{itemize}
		\item \emph{join relations}:
		\begin{itemize}\itemsep0.2cm
			\item if $J_1,J_2$ are strict intervals such that $J_1$ is left adjacent to $J_2$,
			\begin{align}
			K_{J_1}\, K_{J_2}=K_{J_1\cup J_2}\ ;
			\end{align}
			
			\item if $J_1, J_2$ are strict intervals such that $J_1$ is left adjacent to $J_2$ and $J_1 \cup J_2$ is again a strict interval,
			\begin{align}
			\begin{aligned}
			E_{J_1\cup J_2} &=\upsilon^{1/2}\,E_{J_1}\,E_{J_2} - \upsilon^{-1/2}\,E_{J_2}\, E_{J_1}\ ,  \\[2pt]
			F_{J_1\cup J_2} &=\upsilon^{-1/2}\,F_{J_2}\,F_{J_1} - \upsilon^{1/2}\, F_{J_1}\, F_{J_2}  \ ;
			\end{aligned}
			\end{align}
		\end{itemize}
		
		\item \emph{nest relations}:
		\begin{itemize}\itemsep0.2cm
			\item if $J_1, J_2$ are strict intervals such that $\overline{J_1} \cap \overline{J_2} =\emptyset$,
			\begin{align}
			[E_{J_1},E_{J_2}]=0\quad\text{and}\quad [F_{J_1}, F_{J_2}]=0\ ;
			\end{align}
			\item if $J_1, J_2$ are strict intervals such that $J_1\subseteq J_2$,
			\begin{align}
			\begin{split}
			\upsilon^{\langle \cf_{J_1}, \cf_{J_2}\rangle} \, E_{J_1}\, E_{J_2} &=\upsilon^{\langle \cf_{J_2}, \cf_{J_1}\rangle} \, E_{J_2}\, E_{J_1} \ ,\\
			\upsilon^{\langle \cf_{J_1}, \cf_{J_2}\rangle} \, F_{J_1}\, F_{J_2} &=\upsilon^{\langle \cf_{J_2}, \cf_{J_1}\rangle} \, F_{J_2}\, F_{J_1} \ .
			\end{split}
			\end{align}
		\end{itemize}
	\end{itemize}
	
	The coproduct is given by
	\begin{align}
	\begin{aligned}
	\Delta(K_J)&=K_J\otimes K_J\ ,\\[2pt]
	\Delta (E_{[a,\, b[})&=E_{[a,\, b[} \otimes 1 + \sum_{a<c<b} \upsilon^{-1/2}\, (\upsilon-\upsilon^{-1})\, E_{[a,\, c[}\, K_{[c,\, b[} \otimes E_{[c,\, b[} + K_{[a,\, b[} \otimes E_{[a,\, b[} \ ,\\[2pt]
	\Delta (F_{[a,\, b[})&=1\otimes F_{[a, \, b[} - \sum_{a<c<b} \upsilon^{-1/2}\, (\upsilon-\upsilon^{-1})\, F_{[c,\, b[}\otimes F_{[a,\, c[}\, K_{[c,\, b[}^{-1}  + F_{[a,\, b[}\otimes K_{[a,\, b[}^{-1}\ .
	\end{aligned}
	\end{align}
	Here the sums on the right-hand-side run over all possible $\K$-values $c\in [a, b[$.
\end{definition}

\begin{definition}
	We define a $\fun{S^1_\K}$-gradation on $\bfU_\upsilon(\fraksl(S^1_\K))$ by setting
	\begin{align}
	\deg(E_J)=\cf_J\ ,\quad \deg(F_J)=-\cf_J \quad\text{and}\quad \deg(K_{J'})=0
	\end{align}
	for any strict interval $J$ and any interval $J'$.
\end{definition}
As in Section~\ref{ss:algebra-K}, we define the subalgebras $\bfU_\upsilon^\epsilon(\fraksl(S^1_\K))$, with $\epsilon=0, \pm 1$, as well as the negative subalgebra $\bfU_\upsilon^{\leq 0}(\fraksl(S^1_\K))$ and the positive subalgebra $\bfU_\upsilon^{\geq 0}(\fraksl(S^1_\K))$.

\subsection{Representations of continuum quivers}\label{ss:representations}

Let us begin with the line $\K$, with $\K\in \{\R, \Q\}$ as before. We start by recalling the notion of persistence modules (see  \cite{CBpersistent},  \cite[Section~2.3]{dufresne_sampling} for a recent report).

\begin{definition}\label{def:persistence-module}
	Let $k$ be a field. A \emph{$\K$-persistence module}\footnote{In the literature of persistent homology \cite{carlsson-09, oudot-15}, this is called a \emph{pointwise finite-dimensional} persistence module.} is a functor $F\colon (\K, \leq)\to \mathbf{Vect}_k$ from the poset $(\K, \leq)$ to the category $\mathbf{Vect}_k$ consisting of finite-dimensional vector spaces over $k$ with linear maps between them. Explicitly, $F$ is determined by:
	\begin{itemize}\itemsep0.2cm
		\item[i)] a finite-dimensional $k$-vector space $F(t)$ for every $t \in \K$,
	
		\item[ii)] a $k$-linear map $F(s\leq t)\colon V_s \to V_t$ for every pair of real numbers $s\leq t$ such that
		\begin{itemize}\itemsep0.2cm
			\item $F(t\leq t)$ is the identity map from $V_t$ to itself,
	
			\item given real numbers $s\leq t\leq u$, one has $F(s\leq u)=F(t\leq u)\circ F(s\leq t)$.
		\end{itemize} 
	\end{itemize}
\end{definition}
Morphisms of $\K$-persistence modules are defined in the obvious way.

\begin{remark}\label{rem:generalization}
	The poset $(\K, \leq)$ can be interpreted as a category, whose objects are points of $\K$ and the set of maps between two objects $s,t\in \K$ consists of exactly one map $\iota_{t,\, s}$ if $s\leq t$, otherwise is empty. Thanks to this interpretation, one can extend the definition above replacing $(\K, \leq)$ with any (small) category, as done e.g. in \cite[Section~2]{CBpersistent}.
\end{remark}

\begin{definition}
	A point $t\in \K$ is \emph{regular} for a persistence module $F$ if there exists an interval $I\subseteq \K$ where $t\in I$ and $F(a\leq b)$ is an isomorphism for all pairs $a, b\in I$. Otherwise we say that $t$ is \emph{critical}. A persistence module $F$ is \emph{tame} if it has finitely many critical values.
\end{definition}
The set of all regular points will be called the \emph{regular locus} and its complement the \emph{critical locus}.  

\begin{example}\label{ex:k_J}
	Any $\K$-interval $J=[a,b[$ gives rise to a persistence module $k_J$, for which $k_J(t)=k$ if $t\in J$, otherwise $k_J(t)=\{0\}$, and for any pair $s,t\in \K$ one has that $k_J(s\leq t)$ is the identity map if $s,t \in J$, otherwise $k_J(s\leq t)$ is the zero map. By \cite[Proposition~2.2]{BLpersistent}, $k_J$ is an example of an indecomposable module with local endomorphism ring.
\end{example}	

From now on, we shall restrict to a smaller class of persistence modules.
\begin{definition}
	We say that a persistence module $F$ is \emph{coherent} if it is tame and the map $t \mapsto \dim F(t)$ is right--continuous and compactly supported.
\end{definition}

The fundamental theorem of persistent homology says that any (pointwise finite-dimensional) persistence module is a direct sum of indecomposable modules with local endomorphism ring (cf.\  \cite[Theorem~1.1]{CBpersistent}). If we apply it to our coherent persistence modules, we obtain:
\begin{lemma}\label{lem:fundamental}
	For any coherent $\K$-persistence module $F$ there exists a partition
	\begin{align}
		\{x\in \K \;\vert\; F(x)\neq \{0\}\}=\bigsqcup_{i} J_i
	\end{align}
	into finitely many $\K$-intervals $J_i=[a_i,b_i[$ such that $F(v\leq u)$ is an isomorphism for any pair $v\leq u$ of elements belonging to the same interval $J_i$.
\end{lemma}
 
From now on, we shall also call a coherent persistence module a \emph{representation of the continuum quiver $\K$}. We denote the data of a representation $F$ of $\K$ in the following way: $V_t\coloneqq F(t)$ and $x_{t,\, s}\coloneqq F(s\leq t)$.  Summarizing, a representation of $\K$ is a collection of data $(V, x)\coloneqq (V_t, x_{t,\, s})_{t,\, s}$ with
\begin{enumerate}\itemsep0.2cm
	\item[i)] $V_t$ is a finite-dimensional $k$-vector space for every $t \in \K$,
	\item[ii)] $x_{t,\, s}\colon V_s \to V_t$ is a $k$-linear map for every pair $s,t \in \K$ with $s\leq t$,
\end{enumerate}
such that
\begin{enumerate}\itemsep0.2cm
	\item[a)] the map $t \mapsto \dim(V_t)$ is right--continuous, compactly supported and with finitely many discontinuities,
	\item[b)] we have $x_{t,\, r}=x_{t,\, s} \circ x_{s,\, r}$ for any triple $r \leq s \leq t$,
	\item[c)] there exists a partition $\{t\in \K \;\vert\; V_t\neq \{0\}\}=\bigsqcup_{i} J_i$ into finitely many intervals $J_i=[a_i,b_i[$ such that $x_{u,v}$ is an isomorphism for any pair $v\leq u$ of elements belonging to the same interval $J_i$.
\end{enumerate}

We denote by $\Rep_k\K$ the category of representations of $\K$. Then $\Rep_k\K$ is abelian. It may be realized as a certain colimit of categories of representations of (locally) finite type $A$ quivers. More precisely, for any locally finite subset $S \subset \K$, let us denote by $\Rep^{(S)}_k\K$ the full subcategory of $\Rep_k\K$ consisting of representations $(V_t,x_{t,\, s})_{t,\, s}$ for which $x_{t,\, s}$ is an isomorphism for $s, t\in \K\setminus S$. Then $\Rep^{(S)}_k\K$ is equivalent to the category $\Rep_k\calQ_S$ of representations of the quiver $\calQ_S$, whose vertices are the maximal intervals in $\K \setminus S$ and whose arrows are given by the adjacency relation. It is clear that $\calQ_S$ is either a finite type $A$ quiver (if $|S| <\infty$) or a quiver of type $A_{\infty}$ (if $S$ is infinite). 

The categories $\Rep^{(S)}_k\K$ form a direct system with respect to the inclusion $S \subset S'$ and we have
\begin{align}
	\Rep_k\K \simeq \varinjlim \Rep_k\calQ_S\ .
\end{align}
It follows that $\Rep_k\K$ is hereditary.

\begin{rem}
	The above equivalence justifies our choice to call the objects of $\Rep_k\K$ the \emph{representations of $\K$}.\hfill $\triangle$
\end{rem}

Let us next consider the case of the circle $S^1_{\K}$. Let $\Gamma$ be the oriented fundamental groupoid of $S^1_{\K}$: its objects are homotopy classes of orientation preserving paths $[0,1]\cap \K \to S^1_{\K}$; these may be parametrized by triples $(s, t, n)$ where $s,t$ are (not necessarily distinct) elements $S^1_{\K}$ and $n \in \N$ is the winding number of the path. For $\gamma= (s,t, n)$ we will sometimes write $\gamma'=s, \gamma''=t$. 

By following Remark~\ref{rem:generalization}, we replace $(\K, \leq)$ in Definition~\ref{def:persistence-module} by the category whose objects are points $t\in S^1_\K$ and morphisms are given by $\gamma\in \Gamma$: thus, the notion of coherent persistence modules makes sense. We denote by $\Rep_{k}S^1_\K$ the category of these coherent persistence modules, which we shall call \emph{representations of the continuum quiver $S^1_\K$}. Note that Lemma~\ref{lem:fundamental} holds also in this case, thanks to \cite[Theorem~1.1]{BLpersistent}. Summarizing, a representation of $S^1_\K$ is a collection of data $(V, x)\coloneqq (V_t, x_{\gamma})_{t,\gamma}$ with $t\in S^1_\K$ and $\gamma\in \Gamma$ with
\begin{enumerate}\itemsep0.2cm
	\item[i)] $V_t$ is a finite-dimensional $k$-vector space for every $t \in S^1_{\K}$,
	\item[ii)]$x_{\gamma}\colon V_{\gamma'} \to V_{\gamma''}$ is a $k$-linear map for every $\gamma \in \Gamma$, 
\end{enumerate}
such that
\begin{enumerate}\itemsep0.2cm
	\item[a')] the map $t \mapsto \dim(V_t)$ is right--continuous and with finitely many discontinuities,
	\item[b')] we have $x_{\gamma_1} \circ x_{\gamma_2}=x_{\gamma_1 \cdot \gamma_2}$ for any composable pair $\gamma_1, \gamma_2$ of elements of $\Gamma$,
	\item[c')] there exists a partition $S^1_\K=\bigsqcup_{i} J_i$ into finitely many intervals $J_i=[a_i,b_i[$ such that $x_{\gamma}$ is an isomorphism for any path $\gamma$ contained in a single interval $J_i$.
\end{enumerate}

Again, $\Rep_kS^1_{\K}$ is a hereditary abelian category. Like $\Rep_k\K$ it may be realized as a colimit $\Rep_kS^1_\K=\varinjlim \Rep^{(S)}_kS^1_{\K}$, where $S$ ranges over all finite subsets of $S^1_\K$. Note that $\Rep^{(S)}_kS^1_\K$ is now equivalent to a category $\Rep_k \calQ_S$ of representations of an affine (cyclically oriented) quiver of type $A$. For any strict interval $J=[a,b[$ of $S^1_\K$ there is a indecomposable object $k_J$ of $\Rep_k S^1_\K$, and one obtains in this way all rigid indecomposables.

Taking the dimension function $f(V)\colon t \mapsto \dim(V_t)$ defines a map $\underline{\dim}$ from the Grothendieck groups $\mathsf{K}_0(\Rep_k\K)$ and $\mathsf{K}_0(\Rep_kS^1_\K)$ to $\fun{\K}$ and $\fun{S^1_\K}$ respectively. 

\begin{remark}
	Let $X$ be a scheme, $L$ a line bundle on $X$ and $s$ a global section of $L$. Following \cite[Example~3.18]{art:talpo2018}, a \emph{parabolic sheaf with real weights} on $X$  is the assignment of an $\calO_X$-module $E_r$ for each $r\in \R$, with maps $E_r\to E_{r'}$ when $r\leq r'$, that are compatible with respect to composition, and such that $E_r\to E_{r+1} \simeq L \otimes_{\calO_X} E_r$ is identified with multiplication by the section $s$. 
	
	Given a real number $r$, we denote by $\lfloor r\rfloor$ the largest integer which is less or equal to $r$ and we set $\{r\} \coloneqq r − \lfloor r\rfloor$. A parabolic sheaf with real weights $E$ is \emph{finitely presented} if each $E_r$ is a finitely presented sheaf on $X$, and moreover there exist finitely many real numbers $0 \leq r_1 < \ldots < r_k < 1$, such that for every $r \in \R$ then $E_r \simeq E_{r_i}	\otimes L_{\lfloor r\rfloor}$ via the given map, where $r_i$ is the largest of the fixed numbers that is less or equal to $\{r\}$. In other words, the parabolic sheaf E is completely determined by the weights $r_i$, the finitely presented sheaves $E_{r_i}$, and the maps between them. 
	
	Then, our definition of representations of $S^1_\R$ coincides with the definition of finitely presented parabolic sheaves $E$ for which each $E_r$ has zero-dimensional support for any $r\in \R$.
\end{remark}

\subsection{Hall algebras of continuum quivers}

When $k=\F_q$ is a finite field, $\Rep_k\K$ and $\Rep_kS^1_\K$ are finitary categories, i.e., the sets $\mathsf{Ext}^i(V,V')$ are finite for any $V,V'$ and $i=0,1$. In this context we may consider the Hall algebras $\bfH_\K$ and $\bfH_{S^1_\K}$ of $\Rep_k\K$ and $\Rep_kS^1_\K$ respectively, as in Section~\ref{ss:VVbusiness}.
Recall that  $\upsilon\coloneqq q^{1/2}$. Let $\calM_\K, \calM_{S^1}$ denote the set of all isomorphism classes of objects of $\Rep_k\K$ and $\Rep_kS^1_K$ respectively. As vector spaces we have
\begin{align}
	\bfH_{\K}\coloneqq \{ h\colon \calM_\K \to \C\;|\; |\mathsf{supp}(f)| <\infty\} \quad\text{and}\quad \bfH_{S^1_\K}\coloneqq\{ h\colon \calM_{S^1} \to \C\;|\; |\mathsf{supp}(f)| <\infty\}\ .
\end{align}
As before, denote by $[V]$ the characteristic function of an object $V$. The function $\underline{\dim}$ induces a natural gradation
\begin{align}\label{E:defHall}
\bfH_{\K}=\bigoplus_{f \in \fun{\K}} \bfH_\K[f] \quad\text{and}\quad \bfH_{S^1_\K}=\bigoplus_{f \in \F(S^1_K)} \bfH_{S^1_\K}[f]\ .
\end{align}

The multiplications in $\bfH_\K$ and $\bfH_{S^1_\K}$ are defined in the same fashion as in Section~\ref{ss:VVbusiness}. Note that the multiplication is graded with respect to the decomposition \eqref{E:defHall}. 

We let $\bfH^{\mathsf{sph}}_{S^1_\K}$ be the subalgebra of $\bfH_{S^1_\K}$ generated by the elements $[k_J]$ where $J$ runs through the set of \textit{strict} intervals in $S^1$.

\begin{theorem}
\hfill
\begin{itemize}\itemsep0.2cm
	\item[(i)] The assignment $[k_{J}] \mapsto \upsilon^{-1/2}E_J$ for all intervals $J$ defines an isomorphism of graded algebras $\mathbf{H}_{\mathbb{K}} \simeq \bfU^+_{\upsilon}(\fraksl(\K))$.
	\item[(ii)] The assignment $[k_{J}] \mapsto \upsilon^{-1/2}E_J$ for all strict intervals $J$ defines an isomorphism of graded algebras $\mathbf{H}^{\mathsf{sph}}_{S^1_\K} \simeq \bfU^+_{\upsilon}(\fraksl(S^1_\K))$. 
\end{itemize}	
\end{theorem}

\begin{proof} 
	Notice that any order-preserving bijection $\K \to ]0,1[ \cap \K$ induces both a fully faithful embedding $\Rep_k\K \hookrightarrow \Rep_kS^1_\K$ and an inclusion of algebras $\bfU^+_\upsilon(\fraksl(\K)) \hookrightarrow \bfU^+_\upsilon(\fraksl(S^1_{\K}))$. Hence, by the compatibility between the Hall algebra construction and fully faithful embeddings (see e.g. \cite[Section~1.8]{art:schiffmann-lectures}) it is enough to treat the case of $\Rep_kS^1_{\K}$ and $\bfU^+_\upsilon(\fraksl(S^1_\K))$. By the same functorial properties of Hall algebras, $\bfH^{\mathsf{sph}}_{S^1_\K}$ is isomorphic to an inductive limit of Hall algebras $\bfH^{\mathsf{sph}}(\Rep^{(S)}_k(S^1_{\K}))$ as $S$ ranges over all finite subsets of $S^1_\K$. On the other hand, denote by $\bfU^+_\upsilon(\fraksl^{(S)}(S^1_{\K}))$ the subalgebra of $\bfU^+_\upsilon(\fraksl(S^1_{\K}))$ generated by the elements $E_J$ for $J =[a,b[$ satisfying $a,b \in S$. Then $\bfU^+_\upsilon(\fraksl(S^1_{\K}))$ is the inductive limit of $\bfU^+_\upsilon(\fraksl^{(S)}(S^1_{\K}))$ as $S$ ranges over all finite subsets of $S^1_\K$. Hence it suffices to prove that the assignment $E_J \mapsto \upsilon^{1/2}[k_J]$ extends to an algebra isomorphism $\bfU^+_\upsilon(\fraksl^{(S)}(S^1_{\K}))\xrightarrow{\sim} \bfH^{\mathsf{sph}}(\Rep^{(S)}_k(S^1_{\K}))$ for any $S$. Since $S$ is finite, this reduces to the well-known identification between the spherical Hall algebra of a cyclic quiver with $n$ vertices and $\bfU^+_\upsilon(\widehat{\fraksl}(n))$ (see e.g. \cite[Section~4.2]{art:salaschiffmann2017}).
\end{proof}

We will need a slight variant of the above result, involving $\bfU^-_{\upsilon}(\fraksl(S^1_\K))$ instead of $\bfU^+_{\upsilon}(\fraksl(S^1_\K))$. 

\begin{corollary}
\hfill
\begin{itemize}\itemsep0.2cm
	\item[(i)] The assignment $[k_{J}] \mapsto -\upsilon^{-1/2}F_J$ for all intervals $J$ defines an isomorphism of graded algebras $\mathbf{H}_{\mathbb{K}} \simeq \bfU^-_{\upsilon}(\fraksl(\K))$.
	\item[(ii)]  The assignment $[k_{J}] \mapsto -\upsilon^{-1/2}F_J$ for all strict intervals $J$ defines an isomorphism of graded algebras $\mathbf{H}^{\mathsf{sph}}_{S^1_\K} \simeq \bfU^-_{\upsilon}(\fraksl(S^1_\K))$.
\end{itemize}
\end{corollary}

\begin{proof} 
	This comes from the isomorphism $\bfU^+_\upsilon(\fraksl(S^1_{\K})) \simeq \bfU^-_{\upsilon}(\fraksl(S^1_\K))$ given by $E_J \mapsto -F_J$.
\end{proof}

\subsection{Folding procedure}

To unburden the notation, let us simply denote by $f \mapsto \overline{f}$ the projection $\pi_\K\colon  \fun{\K} \to \fun{S^1_\K}$ and likewise  $s \mapsto \overline{s}$ for the projection $\K \to S^1_\K$. We will denote by $s,t,\ldots$ elements of $\K$ and by $a,b,\ldots$ elements of $S^1_\K$.

Following Section~\ref{ss:VVbusiness} we will now construct a family of maps 
\begin{align}
	\gamma_f\colon \bfH_{S^1_\K}[\overline{f}] \to \bfH_\K[f]
\end{align}
for every $f\in \fun{\K}$. Fix a function $f \in \fun{\K}$ and a collection of vector spaces $V_t$, for $t \in \K$, such that $\dim(V_t)=f(t)$ for all $t$. Put:
\begin{align}
	V=\bigoplus_tV_t\ , \quad \overline{V}=\bigoplus_{a} \overline{V}_a \ ,\quad \overline{V}_a=\bigoplus_{\overline{s}=a} V_s\ .
\end{align}
Thus, even if they may be canonically identified, $V$ is a $\K$-graded vector space while $\overline{V}$ is $S^1_\K$-graded. Note that since $\mathsf{supp}(f)$ is compact, $\overline{V}_a$ is finite-dimensional for any $a$. For any $s \in \K$ we define
\begin{align}
	V_{\geq s}\coloneqq \bigoplus_{t \geq s} V_t\ ,
\end{align}
and we denote by $V_{\geq}^{\bullet}$ the associated filtration of $V$. There is an induced filtration $\overline{V}_{\geq}^{\bullet}$ of $\overline{V}$ where
\begin{align}
	\overline{V}_{\geq s}=\bigoplus_a \overline{V}_{\geq s, \, a}\quad \text{and}\quad \overline{V}_{\geq s,\,  a}=\bigoplus_{t \geq s,\, \overline{t}=a} V_t\ .
\end{align}

The filtrations $V_{>}^{\bullet},\overline{V}_{>}^{\bullet}$ are defined in the same way, replacing $\geq$ by $>$. The associated graded space
\begin{align}
	\mathsf{gr}(\overline{V})\coloneqq \bigoplus_s (\overline{V}_{\geq s}/ \overline{V}_{>s})
\end{align}
is canonically isomorphic to $V$ as a $\K$-graded vector space.

Let us denote by $\calE_V$ the groupoid of all representations of $\K$ in $V$: objects are collections of maps $x_{t,\, s} \colon V_s \to V_t$ for $s <t$ satisfying conditions (b) and (c) in Section~\ref{ss:representations} and morphisms are given by elements of $\prod_t \mathsf{GL}(V_t)$. We likewise denote by $\calE_{\overline{V}}$ the groupoid of all representations of $S^1_\K$ in $\overline{V}$: objects are collections of maps $x_\gamma \colon \overline{V}_{\gamma'} \to  \overline{V}_{\gamma''}$ for $\gamma \in\Gamma$ satisfying (b') and (c') and morphisms are elements of $\prod_a \mathsf{GL}(\overline{V}_a)$. Finally, let $\calE_{V,\, \overline{V}}$ denote the groupoid of representations in $\calE_{\overline{V}}$ preserving the filtration $\overline{V}^\bullet_{\geq}$: objects are collections $x_\gamma$ as above such that for any $\gamma \in \Gamma$ and any $s \in \K$ satisfying $\overline{s}=\gamma'$ we have $x_\gamma(V_{\geq s}) \subseteq V_{\geq \gamma(s)}$ --  here $\gamma(s)$ is the deck transformation of $s$ 
associated to $\gamma$; morphisms are given by elements of $\prod_a P_a$ where $P_a \subset \mathsf{GL}(\overline{V}_a)$ is the parabolic subgroup of elements preserving the induced filtration $\overline{V}_{\geq}^\bullet \cap \overline{V}_a$. 

There is an obvious functor $j\colon  \calE_{V,\overline{V}} \to \calE_{\overline{V}}$. Passing to the associated graded we also get a functor $p\colon  \calE_{V,\overline{V}} \to \calE_V$ defined as follows. For every $s,t \in \K$ with $s < t$, let $\gamma$ correspond to the oriented path $s \mapsto t$ in $\K$; the map $x_\gamma\colon \overline{V}_{\overline{s}} \to \overline{V}_{\overline{t}}$ sends $\overline{V}_{\geq s}$ to $\overline{V}_{\geq t}$ and $\overline{V}_{> s}$ to $\overline{V}_{> t} $; we let $x_{t,\, s}$ be the induced map $V_s \to \overline{V}_{\geq s} / \overline{V}_{> s} \to  \overline{V}_{\geq t} / \overline{V}_{> t}=V_t$.

\begin{lemma}\label{L:hallproof1} 
	The functor $p \colon \calE_{V,\overline{V}} \to \calE_V$ has essentially finite fibers.
\end{lemma}

\begin{proof}
	Let $\rho=(x_{t,s})_{s,t}\in \calE_V$. Denote by $D_\rho\subset \K$ the set of critical points of $\rho$ and let $\overline{D_\rho} \subset S^1_\K$ be its reduction mod $\Z$. Any object $\tilde{\rho}=(x_\gamma)_\gamma$ in the fiber $p^{-1}(\rho)$ has a critical set $D_{\tilde{\rho}}$ included in $\overline{D_\rho}$. Let $S$ be a finite subset of $S^1_\K$ containing  $\overline{D_\rho}$. Let $\calE^{(S)}_V, \calE^{(S)}_{\overline{V}},\calE^{(S)}_{V,\overline{V}}$ be the (full) subgroupoids of $\calE_V, \calE_{\overline{V}}, \calE_{V,\overline{V}}$ whose objects are those representations whose critical sets are contained in $\pi_\K^{-1}(S)$ and $S$ respectively. The functor $p\colon \calE_{V,\overline{V}} \to \calE_V$ restricts to a functor $p^{(S)}\colon \calE^{(S)}_{V,\overline{V}} \to \calE^{(S)}_V$ and there is a cartesian square
	\begin{align}
		\begin{tikzcd}[ampersand replacement = \&]
			\calE^{(S)}_{V,\overline{V}} \ar[r] \ar{d}{p^{(S)}} \&  \calE_{V,\overline{V}}\ar{d}{p}\\
			\calE^{(S)}_V \ar[r] \& \calE_V
		\end{tikzcd}\ .
	\end{align}
	Since $\rho \in E_V^{(S)}$, it suffices to show that $p^{(S)}$ has essentially finite fibers. This is obvious since $\calE^{(S)}_{V,\overline{V}}$ has finitely many objects up to isomorphism.
\end{proof}

For a groupoid $X$, we denote by $\C(X)$ the space of complex functions on $\mathsf{Obj}(X)$ which are invariant under isomorphism and have essentially finite support. We will identify $\C(\calE_V)$ and $\C(\calE_{\overline{V}})$ with $\bfH_\K[f]$ and $\bfH_{S^1_K}[\overline{f}]$ respectively. We set
\begin{align}
	\gamma_f\coloneqq \upsilon^{-h(f)}\, p_! \circ j^\ast \colon  \bfH_{S^1_\K}[\overline{f}]\to \bfH_\K[f]\ ,
\end{align}
where
\begin{align}
	h(f)\coloneqq -\sum_{\ell<0} \langle f, \tau^\ell(f)\rangle\ ,
\end{align}
with $\tau(f)\colon t \mapsto f(t-1)$ being the translation operator in $\fun{\K}$,
and where $p_!$ is taken in the sense of groupoids, i.e.,
\begin{align}
	p_!(u)(\rho)\coloneqq \sum_{\tilde{\rho}\, \in \mathsf{Obj}(p^{-1}(\rho))/\sim}\, u(\tilde{\rho}) \frac{|\mathsf{Aut}(\rho)|}{|\mathsf{Aut}(\tilde{\rho})|}\ .
\end{align}

\begin{proposition}\label{P:fockhall}
	For any pair $h,h' \in \fun{S^1_\K}$ and any $f \in \fun{\K}$ such that $\overline{f}=h + h'$ we have
	\begin{align}
		\sum_{\genfrac{}{}{0pt}{}{g+g'=f}{\overline{g}=h,\, \overline{g'}=h'}} \, \upsilon^{(g,\sum_{\ell >0}\tau^{\ell}(g'))}\gamma_g(u) \star \gamma_{g'}(u')=\gamma_f(u \star u')
	\end{align}
	for any $u \in \bfH_{S^1_\K}[h], u' \in \bfH_{S^1_\K}[h']$.
\end{proposition}

\begin{proof} 
	Let $f,h,h'$ be as above and fix some $u \in \bfH_{S^1_\K}[h], u' \in \bfH_{S^1_\K}[h']$. Let $S$ be a finite subset of $S^1_\K$ containing the critical sets of all representations occurring in $u,u'$ as well as all the discontinuities of $f,h,h'$. Then $S$ contains also the critical sets of any representation occurring in $u \star u'$, as well as in $\gamma_g(u),\gamma_{g'}(u')$ and hence also of $ \gamma_g(u) \star \gamma_{g'}(u')$. Arguing as in the proof of Lemma~\ref{L:hallproof1}, we see that it is enough to prove the statement of the theorem when we replace everywhere $\calE_V, \ldots$ by $\calE_V^{(S)},\ldots$. But in this case $\mathsf{Rep}_k^{(S)}S^1_\K$ and $\mathsf{Rep}_k^{(S)}\K$ are equivalent to categories of representations of a cyclic quiver and a $A_\infty$ quiver respectively and we are in the precise setting of Proposition~\ref{P:Fockfinite}.
\end{proof}

Using the above Proposition, we may define an algebra morphism from $\bfH_{S^1_\K}$ to a certain completion of $\bfU^{\leq 0}_{\upsilon}(\fraksl(\K))$ that we are going to introduce now. Let $\fun{\K}^c$ be the space of piecewise constant, right--continuous functions $f\colon \R \to \Z$ whose points of discontinuity all belong to $\K$ and project onto a finite subset of $S^1_\K$. In other words, $\fun{\K}^{\mathsf{c}}$ is an analogue of $\fun{\K}$ but without the bounded support condition. Let $\bfU^0_\upsilon(\fraksl(\K))^{\mathsf{c}}$ be the algebra generated by elements $K^\pm_f$ for $f \in \fun{\K}^{\mathsf{c}}$ subject to the relations $K_f K_g=K_{f+g}$, for $f,g \in \fun{\K}^{\mathsf{c}}$. We set
\begin{align}
	\bfU^{\leq 0}_{\upsilon}(\fraksl(\K))^{\mathsf{c}}&\coloneqq \bigoplus_{g \in \fun{S^1_\K}} \bfU^{\leq 0}_{\upsilon}(\fraksl(\K))^{\mathsf{c}}[g] \,\\[4pt]
	\bfU^{\leq 0}_{\upsilon}(\fraksl(\K))^{\mathsf{c}}[g]& \coloneqq \prod_{f \in \fun{\K}, \overline{f}=g}\left\{\bfU^0_\upsilon(\fraksl(\K))^{\mathsf{c}}\ltimes \bfU^-_\upsilon(\fraksl(\K))[f]\right\}\ .
\end{align}
As before, one can show that $\bfU^{\leq 0}_{\upsilon}(\fraksl(\K))^{\mathsf{c}}$ is an algebra (for this, we use that any $g \in \F(\K)$ can be written in finitely many ways as a sum $g=g'+g''$ with $\overline{g},\overline{g'}$ fixed). For $x \in \bfH_{S^1_\K}[g]$ we set
\begin{align}
	r(x)\coloneqq \sum_{f:\, \overline{f}=g} K_{o(f)} \, \gamma_f(u) \in \bfU^{\leq 0}_{\upsilon}(\fraksl(\K))^{\mathsf{c}} 
\end{align}
where $o(f)\coloneqq \sum_{\ell <0} \tau^{\ell}(f)$. 

\begin{corollary} 
	The map $r\colon \bfH_{S^1_\K} \to \bfU^{\leq 0}_{\upsilon}(\fraksl(\K))^{\mathsf{c}}$ is an algebra homomorphism.
\end{corollary}

One can show that the map $r$ is injective by realizing $\bfH_{S^1_\K}$ as a direct limit of Hall algebras of cyclic quivers and using the injectivity of the map $r$ in the finite case.

\begin{remark}
	Using the above Proposition, we can pull-back via $r$ any weight representation of $\bfU^{\leq 0}_{\upsilon}(\fraksl(\K))$ which extends to $\bfU^{\leq 0}_{\upsilon}(\fraksl(\K))^{\mathsf{c}}$. In the next section, we will define the Fock space representation of $\bfH_{S^1_\K}$ or $\bfU^-_\upsilon(\fraksl(S^1_\K))$ as the pullback by $r$ of the Fock space representation $\calF_\K$ of $\bfU^{\leq 0}_{\upsilon}(\fraksl(\K))$. 
\end{remark}

Let us compute explicitly the image of the elements $F_J$.
\begin{lemma} 
	Let $J$ be a strict interval of $S^1_\K$. Then
	\begin{align}\label{E:nezt}
	r(F_J)=\sum_{J_1, \ldots, J_\ell}\, \upsilon^{\frac{\ell-1}{2}}(\upsilon^{-1}-\upsilon)^{\ell-1}\, F_{J_1} \cdots F_{J_\ell} \prod_{i, \ell>0} K_{\cf_{J_i+\ell}}
	\end{align}
	where the sum ranges over all tuples $J_1, J_2, \ldots, J_\ell$ such that $\cf_J=\pi(\cf_{J_1}+ \cdots +\cf_{J_\ell})$, $J_1 < J_2 < \cdots$ and $\pi_\K(J_1) \to   \pi_\K(J_2) \to \cdots$. 
\end{lemma}

\begin{proof} 
	Let $f \in \fun{\K}$ be such that $\overline{f}=\cf_J$. We claim that
	\begin{align}
		\gamma_f(F_J)=\upsilon^{\frac{1-\ell}{2}}(\upsilon^{-1}-\upsilon)^{\ell-1}\, F_{J_1} \cdots F_{J_\ell}
	\end{align}
	if there exists $J_1, J_2, \ldots, J_\ell$ such that $f=\cf_{J_1}+ \cdots +\cf_{J_\ell}$, $J_1 < J_2 < \cdots$ and $\pi_\K(J_1) \to   \pi_\K(J_2) \to \cdots$ and $\gamma_f(F_J)=0$ otherwise. Let us write $f=\sum_{i=1}^{\ell} \cf_{J_i}$ with $J_1 < J_2 < \cdots J_\ell$. For $\gamma_f(F_J)$ to be nonzero, there must exist a filtration $M_1 \subset M_2 \subset \cdots \subset M_\ell= k_J$ of the indecomposable $S^1_\K$-module $k_J$ such that $\dim(M_i/M_{i-1})=\overline{\cf_{J_i}}$ for $i=1, \ldots, \ell$. This is possible if and only if  
	$\pi_\K(J_1) \to   \pi_\K(J_2) \to \cdots$ and moreover in that case we have $\bigoplus_i M_i/M_{i-1} \simeq \bigoplus_i k_{J_i}$. The claim now follows from the facts that $|\mathsf{Aut}(k_J)|=q-1$, while $|\mathsf{Aut}(\bigoplus_i I_{J_i})|=(q-1)^\ell$ and $h(f)=1-\ell$, and finally from the easily checked relation $[\bigoplus_i k_{J_i}]= [k_{J_1}] \star \cdots \star [k_{J_\ell}]$. To obtain formula \eqref{E:nezt}, observe that for $i<k$ we have $(\tau^{\ell}(\cf_{J_k}),\cf_{J_i}) =0$ for all $\ell>0$ while
	for $i>k$ there exists a unique $\ell$ such that $(\tau^\ell(\cf_{J_k}),\cf_{J_i}) =1$ and $(\tau^\ell(\cf_{J_k}),\cf_{J_i}) =0$ for all other values of $\ell$.
\end{proof}

\bigskip\section{Fock space representation of the circle quantum group}\label{s:Fockspacecircle}

In this section, we define an action of the circle quantum group on the Fock space $\calF_\K$. 

\subsection{The Fock space $\calF_{\K}$ of $\bfU_\upsilon(\fraksl(S^1_\K))$}\label{ss:Fockspace-circle}

We will now define the main object of the present paper, namely the Fock space representation of the circle quantum group. As a vector space, this Fock space is again
\begin{align}
	\calF_{\K}=\bigoplus_{p \in  \mathsf{Pyr}(\K)}\,  \widetilde \Q \vert p\rangle\ .
\end{align}

For $I$ a $\K$-interval and $p$ a $\K$-pyramid, we set
\begin{align}
	n^>_I(p)=\sum_{m \geq 1} n_{\tau_m(I)}(p), \qquad n^<_I(p)=\sum_{m \geq 1} n_{\tau_{-m}(I)}(p), \qquad \overline{n}_I(p)=\sum_{m \in \Z} n_{\tau_m(I)}(p)\ ,
\end{align}
where $n_J(p)$ is defined as in Theorem~\ref{thm:Fock-R} and where $\tau_m(I)$ is the translation of $I$ by $m\geq 1$ to the right, i.e. if $I=[a,b)$ then $\tau_m(I)=[a+m,b+m)$. If $I, J$ are $\K$-intervals then we write $I < J$ if $I=[a,b), J=[c,d)$ and $b <c$.

\begin{theorem} \label{thm:Fock-S1}
	The following formulas define an action of the quantum group $\bfU_\upsilon(\fraksl(S^1_\K))$ on $\calF_\K$:
	\begin{align}\label{E:thmmain}
		E_J \vert p\rangle&=\sum_{J'_1, \ldots, J'_\ell} \upsilon^{\frac{1-\ell}{2}-\sum_{i} n^<_{J'_i}(p)} (-\upsilon)^{-\sum_i \langle \cf_{J'_i},p\rangle}(\upsilon-\upsilon^{-1})^{\ell-1}
		\vert p-\sum_i \cf_{J'_i}\rangle\ ,\\[4pt]
		\label{E:thmmain2}
		F_J \vert p\rangle&=\sum_{J''_1, \ldots, J''_\ell} \upsilon^{\frac{\ell-1}{2}+\sum_{i} n^>_{J''_i}(p)} (-\upsilon)^{\sum_i \langle \cf_{J''_i},p\rangle}(\upsilon^{-1}-\upsilon)^{\ell-1}
		\vert p+\sum_i \cf_{J''_i}\rangle\ ,\\[4pt]
		K_J \vert p \rangle&= \upsilon^{\overline{n}_{J}(p)}\vert p \rangle\ ,
	\end{align}
	where the sums range over all tuples of removable $\K$-intervals $(J'_1, \ldots, J'_l)$ (resp.\ all tuples of addable $\K$-intervals $(J''_1, \ldots, J''_l)$) satisfying the conditions
	\begin{itemize}\itemsep0.2cm
		\item[a')] $J'_1 > J'_2 > \cdots > J'_\ell$,
		\item[b')] $\pi_\K(J'_1) \to \pi_\K(J'_2) \to \cdots \to \pi_\K(J'_\ell)$,
		\item[c')] $\pi_{\K}(J'_1) \sqcup \cdots \sqcup \pi_\K(J'_\ell)=J$.
	\end{itemize}
	(resp.\
	\begin{itemize}\itemsep0.2cm
		\item[a'')] $J''_1 < J''_2 < \cdots < J''_\ell$,
		\item[b'')] $\pi_\K(J''_1) \to \pi_\K(J''_2) \to \cdots \to \pi_\K(J''_\ell)$,
		\item[c'')] $\pi_{\K}(J''_1) \sqcup \cdots \sqcup \pi_\K(J''_\ell)=J$.)
	\end{itemize}
	
\end{theorem}

\begin{proof}
	The proof of this theorem occupies the rest of the Section.
	
	\subsubsection{Well-definedness}\label{ss:well-def}
	
	Let us first check that $E_J$ and $F_J$ are well-defined, i.e. that the sums involved in their definition are in fact finite. We treat the case of the operators $F_J$, the case of $E_J$ being similar. 
	
	Let us fix a $\K$-pyramid $p$. By definition, a collection of $\K$-intervals $J''_1, J''_2, \ldots, J''_\ell$ satisfying the conditions (a''), (b''), (c'') in Theorem~\ref{thm:Fock-S1} induce a subdivision $J_1 \to J_2 \to \cdots \to J_\ell$ of $J$, with $J_i=\pi_\K(J''_i)$. For any given interval $I \subset S^1_\K$ and any pyramid $q$, there are at most finitely many $I'' \in \Int(I)$ which are addable to $q$ (because $q$ is of compact support). Hence it is enough to prove that only finitely many subdivisions $J_1 \to J_2 \to \cdots \to J_\ell$ of $J$ may give rise to a collection of addable intervals $(J''_1, \ldots, J''_\ell)$. 
	
	Write $J=[a,b)$ and $J_1=[c_0=a, c_1), J_2=[c_1,c_2), \ldots, J_\ell=[c_{\ell-1},c_\ell=b)$. We claim that if there exists addable $\K$-intervals $J''_1, \ldots, J''_\ell$ such that $J''_i \in \Int(J_i)$ for all $i$ and $J''_1 < J''_2 < \cdots < J''_\ell$ then necessarily $c_1, c_2, \ldots, c_{\ell-1} \in \pi_\K(D_\K(p))$. Indeed, writing $J''_i=[a''_i,b''_i)$ for $i=1, \ldots, \ell$ (so that $\pi_\K(a''_i)=c_{i-1}, \pi_\K(b''_i)=c_i$) we have by Remark~\ref{rem:action-nonzero}
	\begin{align}
	&b''_i \in D_\K(p)\quad \text{if $b''_i<0$ and $i \neq \ell$} \ ,\\[2pt]
	&a''_i \in D_\K(p)\cup \{0\}\quad \text{if $a''_i\geq0$ and $i \neq 1$}\ .
	\end{align}
	Since $D_\K(p)$ is finite, this implies the desired finiteness of possible tuples $(J_1, \ldots, J_\ell)$. The argument for $E_J$ is similar, using the \emph{reverse} condition $J'_1> J'_2 > \cdots > J'_\ell$.
	
	\subsubsection{Join and nest relations}
	
	Using Formula~\eqref{E:nezt} and Theorem~\ref{thm:Fock-R} we deduce that the operators \eqref{E:thmmain2} indeed define an action of $\bfU^+_\upsilon(\fraksl(S^1_\K))$ as wanted. To prove that the operators \eqref{E:thmmain} define an action of $\bfU^-_\upsilon(\fraksl(S^1_\K))$, one may argue in a similar fashion --- using Hall algebra of the opposite quivers with $\upsilon^{-2}=q$ --- and show that the assignment  
	\begin{align}
		E_J \mapsto \sum_{J_1, \ldots, J_\ell}\, \upsilon^{\frac{1-\ell}{2}}(\upsilon-\upsilon^{-1})^{\ell-1}\, E_{J_1} \cdots E_{J_\ell} \prod_{i, \ell >0} K^{-1}_{\cf_{J_i-\ell}}
	\end{align}
	where the sum ranges over all tuples $J_1, \ldots, J_\ell$ of intervals such that $\cf_J=\pi(\cf_{J_1}+ \cdots +\cf_{J_\ell})$, $J_1 > J_2 > \cdots$ and $\pi_\K(J_1) \to   \pi_\K(J_2) \to \cdots$, defines an algebra morphism $\bfU^-_\upsilon(\fraksl(S^1_\K)) \to \bfU^-_\upsilon(\fraksl(\K))^{\mathsf{c}}$.
	
	This argument proves the join and nest relations for the $E_J$'s and the $F_J$'s.
	
	\subsubsection{The Drinfeld-Jimbo relations}
	
	It remains to check that the actions of $\bfU_\upsilon^{\pm}(\fraksl(S^1_\K))$ and of $\bfU_\upsilon^{0}(\fraksl(S^1_\K))$ glue together to form an action of $\bfU_\upsilon(\fraksl(S^1_\K))$. Each of the relations (\ref{eq:DJ-1} -- \ref{eq:DJ-5}) involves finitely many generators $E_J, F_J, K_J$ acting on a pyramid $p$. 
	
	Let $\alpha \colon \Z \to \R$ be a strictly increasing map, whose image is invariant under integer translations, contains $0$ as well as $D_\K(p)$ and the endpoints of all involved intervals $J$. Put $N=\# (\alpha^{-1}([0,1)))$. Arguing as in Section~\ref{ss:gen} we construct natural isomorphism $\bfU_\upsilon(\fraksl(S^1_{\frac{1}{N}\Z})) \simeq \bfU_{\upsilon}(\fraksl(S^1_{\alpha(\Z)}))$ and $\calF_{\frac{1}{N}\Z} \simeq \calF_{\alpha(\Z)}$ compatible with the formulas in Theorem~\ref{thm:Fock-S1}. Thus it is enough to prove the theorem in the case $\K=\frac{1}{N}\Z$, which we assume from now on. 
	
	Relations \eqref{eq:DJ-1}, \eqref{eq:DJ-2}, \eqref{eq:DJ-3} are obvious. Using the join relations, we may reduce relations \eqref{eq:DJ-4}, \eqref{eq:DJ-5} to the cases where $J_1, J_2$, resp.\ $J$ are of length $\frac{1}{N}$; this is clear for \eqref{eq:DJ-1} and results from a lengthy but straightforward computation for \eqref{eq:DJ-2}. In the case of length $\frac{1}{N}$-intervals (i.e. of simple roots of $\widehat{\fraksl}(N)$) these relations are well-known; we nevertheless derive them below for the reader's comfort: let $p$ be a $\frac{1}{N}\Z$-pyramid, and let $J=[i,i+\frac{1}{N}[$ for some $i \in \frac{1}{N}\Z /\Z$. Let $p'$ be another $\frac{1}{N}\Z$ pyramid. By construction, we have $\langle p' \;|\; E_JF_J p \rangle =\langle p' \;|\; F_JE_J p \rangle =0 $ unless $p'=p-\cf_{J'} + \cf_{J''}$ for a pair of intervals $J,J'$ such that $\pi(J')=\pi(J')=J$. Let us first show that if $p \neq p'$ then $\langle p'\;|\; F_JE_J p\rangle=\langle p'\;|\; E_JF_J p\rangle$. In that situation, $J'$ and $J''$ are unique and we have
	\begin{align}
		\langle p'\;|\;E_JF_J p\rangle &= \langle p'\;|\; E_J (p+\cf_{J''})\rangle \langle p+\cf_{J''}\;|\; F_J p\rangle= \upsilon^{n^>_{J''}(p)-n^<_{J'}(p+\cf_{J''})}\ ,\\[4pt]
		\langle p'\;|\;F_JE_J p\rangle & = \langle p'\;|\; F_J (p-\cf_{J'})\rangle \langle p-\cf_{J'}\;|\; E_J p\rangle= \upsilon^{n^>_{J''}(p-\cf_{J'})-n^<_{J'}(p)}\ . 
	\end{align}
	
	There are two possibilities: $J'' > J'$ or $J''< J'$. In the first one we have $n^>_{J''}(p)=n^>_{J''}(p-\cf_{J'})$ and $n^<_{J'}(p+\cf_{J''})=n^<_{J'}(p)$, while in the second one we have $n^>_{J''}(p)=n^>_{J''}(p-\cf_{J'})+1$ and $n^<_{J'}(p+\cf_{J''})=n^<_{J'}(p)-1$. In both cases the desired equality follows. 
	
	Thus only when $p=p'$ may we have a contribution to the commutator $[E_J,F_J]$. Let $I_1< I_2< \cdots < I_c$ be the different addable or removable intervals of $p$ which are congruent to $J$. We may partition $[1, c]=A \sqcup R$ with $A=\{i\;|\; I_i\;\text{is addable}\}$ and $R=\{i\;|\; I_i\;\text{is removable}\}$. Let us also write 
	\begin{align}
		a_{>h}=| A \cap [h+1,n] |\ , \; r_{>h}=| R \cap [h+1,n]|\ , \; a_{<h}=| A \cap [1,h-1] |\ , \; r_{<h}=| R \cap [1,h-1]|\ .
	\end{align}
	
	The contribution of the interval $I_h$ to $\langle p\;|\; [E_J,F_J] p\rangle$ is equal to
	\begin{align}
		\langle p\;|\; E_J (p+\cf_{I_h})\rangle \langle p+\cf_{I_h}\;|\; F_J p\rangle= \upsilon^{(a_{>h}-r_{>h})-(a_{<h}-r_{<h})}
	\end{align}
	if $h \in A$, while it is
	\begin{align}
		-\langle p\;|\; F_J (p-\cf_{I_h})\rangle \langle p-\cf_{I_h}\;|\; E_J p\rangle=-\upsilon^{(a_{>h}-r_{>h})-(a_{<h}-r_{<h})}
	\end{align}
	if $h \in R$. All together we get
	\begin{align}\label{E:DJMR2}
		\langle p\;|\; [E_J,F_J] p\rangle = \sum_{h \in A} \upsilon^{(a_{>h}-r_{>h})-(a_{<h}-r_{<h})}-\sum_{h \in R} \upsilon^{(a_{>h}-r_{>h})-(a_{<h}-r_{<h})}\ .
	\end{align}
	It thus only remains to check that the r.h.s. of relation~\eqref{E:DJMR2} coincides with 
	\begin{align}
	\langle p\;|\; (K_J-K_J^{-1})p\rangle /(\upsilon-\upsilon^{-1})=(\upsilon^{|A|-|R|}-\upsilon^{|R|-|H|})/(\upsilon-\upsilon^{-1})\ .
	\end{align}
	This is a purely combinatorial statement, which may be proved as follows. We first check it when $R=[1, u], A=[u+1,n]$, and then we prove that the r.h.s. of relation~\eqref{E:DJMR2} remains unchanged when we exchange the position of a pair of adjacent elements, one which belongs to $A$ and the other to $R$. We leave the details to the reader.
	
\end{proof}

\subsection{Non-cyclicity}

In the finite setup (i.e., for Hayashi's Fock space of $\bfU_{\upsilon}(\widehat{\fraksl}(n))$) the Fock space is not cyclic. The same holds here:

\begin{proposition} 
	The subspace $\bfU_{\upsilon}(\fraksl(S^1_{\K})) \cdot \vert 0 \rangle$ of $\calF_{{\K}}$ is strict.
\end{proposition}

\begin{proof}
	We claim that the element $\vert \cf_{[0,1[}\rangle$ does not belong to $\bfU_{\upsilon}(\fraksl(S^1_{\K})) \cdot \vert 0 \rangle$. Let us argue by contradiction. Let $u \coloneqq P(F_{J_1}, \ldots, F_{J_s})$ be a linear combination of monomials in generators $F_{J_1}, \ldots, F_{J_s}$ such that $u \cdot \vert 0 \rangle = \cf_{[0,1[} $. Choose a finite subset $\overline{\alpha} \subset S^1_{\K}$ containing all the endpoints of the intervals $J_i$ and let $\alpha\colon \Z \to \K$ be such that $\pi_{\K} (\alpha(\Z))=\overline{\alpha}$. There is a canonical embedding $\bfU_{\upsilon}(\fraksl(S^1_{\alpha(\Z)})) \to \bfU_{\upsilon}(\fraksl(S^1_{\K}))$ whose image contains $u$. Moreover, $\bfU_{\upsilon}(\fraksl(S^1_{\alpha(\Z)}))$ is isomorphic to $\bfU_{\upsilon}(\widehat{\fraksl}(N))$, where $N=|\overline{\alpha}|$, and the restriction of $\calF_{\K}$ to $\bfU_{\upsilon}(\fraksl(S^1_{\alpha(\Z)}))$ contains the Fock space $\calF_{\alpha(\Z)}$ as the subspace spanned by all pyramids $p$ satisfying $D_\K(p) \in \alpha(\Z)$. Hence it is enough to check that $| \cf_{[0,1[}\rangle \not\in \bfU_{\upsilon}(\fraksl(S^1_{\alpha(\Z)})) \cdot |0\rangle$ or equivalently that the element $|\lambda\rangle$ with $\lambda=(N)$  does not belong to the subspace $\bfU_{\upsilon}(\widehat{\fraksl}(N))\cdot |0\rangle$ in Hayashi's Fock space. This last statement may be checked by some simple direct calculations which we leave to the reader.
\end{proof}

\begin{remark}
	The Fock space is a cyclic representation of the Hall algebra $\bfH_{S^1_\K}$. This follows by reduction to the case of cyclic quivers which is treated in \cite{art:varagnolovasserot1999}.
\end{remark}

\subsection{Highest weight vectors}

The vacuum vector $\vert 0 \rangle$ is an obvious highest weight vector. Somewhat surprisingly, it turns out to be the \emph{only} highest weight vector in $\calF_{\K}$ when $\K=\Q$ or $\K=\R$.

\begin{proposition} 
	Assume that $\K=\Q$ or $\K=\R$. We have
	\begin{align}
	(\calF_{\K})^{\bfU^+_{\upsilon}(\fraksl(S^1_{\K}))}=\widetilde{\Q}\vert 0 \rangle\ .
	\end{align}
\end{proposition}

\begin{proof} 
	Let $v\coloneqq \sum_{p} \alpha_p \vert p \rangle$ be a highest weight vector. By homogeneity we may assume that all $p$ for which $\alpha_p \neq 0$ are of the same size $s$. Assume that $s\neq 0$ (i.e., that $v \not\in \widetilde{\Q}\vert 0 \rangle$). Because $\alpha_p$ is nonzero for only finitely many pyramids $p$, the set
	\begin{align}
	D \coloneqq \bigcup_{p,\, \alpha_p \neq 0} D_\K(p)
	\end{align}
	is nonempty and finite. Set $x=\max (D)$. Choose some small $\epsilon >0$ such that
	$[x-\epsilon,x] \cap D=\{x\}$. Let $p$ be a pyramid for which $\alpha_p \neq 0$ and $x \in D_\K(p)$ (hence the support of $p$ is an interval of the form $[a,x]$). Set $J=[x-\epsilon,x)$. By construction, $J$ is removable from $p$. Moreover $p$ is the only pyramid which can contribute to $p-\cf_{J}$ in $E_{J} \cdot v$; this follows from condition (a') in the definition of the action of $E_J$ (cf.\ Theorem~\ref{thm:Fock-S1}) and from the fact that $x-\epsilon \not\in D_\K(q)$ for any pyramid $q$ occuring in $v$. We deduce that $\langle p-\cf_{J} \vert E_{J} \cdot v\rangle=\langle p-\cf_{J} \vert E_{J} \cdot p\rangle \neq 0$, which is in contradiction with the assumption on $v$.
\end{proof}

\begin{remark} 
	Hayashi's Fock space decomposes into a direct sum of highest weight representations. In \cite{SIMRN}, these highest weight vectors are obtained by the action of the center of the Hall algebra on the vacuum vector. In the setting of the continuum quivers $\Q/\Z$ or $\R/\Z$ such a center only appears in a suitable completion (see \cite{art:salaschiffmann2017}). This suggests that a better object to study would be a similar completion of our Fock space. We hope to return to this in the future.
\end{remark}

\bigskip

\providecommand{\bysame}{\leavevmode\hbox to3em{\hrulefill}\thinspace}
\providecommand{\MR}{\relax\ifhmode\unskip\space\fi MR }
\providecommand{\MRhref}[2]{%
	\href{http://www.ams.org/mathscinet-getitem?mr=#1}{#2}
}
\providecommand{\href}[2]{#2}

\end{document}